\documentclass[reqno]{amsart}

\usepackage{amsmath,amssymb,amscd,amsthm,amsxtra,latexsym,dsfont}

\newcommand{\T}{\mathbb{T}}
  \newcommand{\med}{{\mathrm{med}}}
\newcommand{\supp}{{\mathrm{supp}}}

\makeatletter
\DeclareRobustCommand\widecheck[1]{{\mathpalette\@widecheck{#1}}}
\def\@widecheck#1#2{%
  \box\z@\hbox{\m@th$#1#2$}%
  \box\tw@\hbox{\m@th$#1%

    \widehat{%
      \vrule\@width\z@\@height\ht\z@
      \vrule\@height\z@\@width\wd\z@}$}%
  \dp\tw@-\ht\z@ \@tempdima\ht\z@ \advance\@tempdima2\ht\tw@
  \divide\@tempdima\thr@@ \box\tw@\hbox{%

    \raise\@tempdima\hbox{\scalebox{1}[-1]{\lower\@tempdima\box\tw@}}}%
  {\ooalign{\box\tw@ \cr \box\z@}}} \makeatother

\newtheorem{theorem}{Theorem}[section]
\newtheorem{maintheorem}{Theorem} \newtheorem{lemma}[theorem]{Lemma}
\newtheorem{proposition}[theorem]{Proposition}
\newtheorem{remark}[theorem]{Remark} 

\newtheorem{definition}[theorem]{Definition}
\newtheorem{corollary}[theorem]{Corollary}

\title{The Gross-Pitaevskii hierarchy on general rectangular tori}
\author[S.~Herr]{Sebastian Herr} \address{Universit\"{a}t Bielefeld,
  Fakult\"{a}t f\"{u}r Mathematik, Postfach 10 01 31, D-33501
  Bielefeld, Germany} \email{herr@math.uni-bielefeld.de}

\author[V.~Sohinger]{Vedran Sohinger} \address{Eidgen\"{o}ssische
  Technische Hochschule Z\"{u}rich, Departement Mathematik,
  R\"{a}mistrasse 101, 8092 Z\"{u}rich, Switzerland}
\email{vedran.sohinger@math.ethz.ch}

\keywords{Gross-Pitaevskii hierarchy, Nonlinear Schr\"{o}dinger
  equation, BBGKY hierarchy, Bose-Einstein condensation, Irrational
  torus} \subjclass[2010]{35Q55,70E55}

\begin{document}

\begin{abstract}
  In this work, we study the Gross-Pitaevskii hierarchy on general
  --rational and irrational-- rectangular tori of dimension two and
  three. This is a system of infinitely many linear partial
  differential equations which arises in the rigorous derivation of
  the nonlinear Schr\"{o}dinger equation. We prove a conditional
  uniqueness result for the hierarchy.  In two dimensions, this result
  allows us to obtain a rigorous derivation of the defocusing cubic
  nonlinear Schr\"{o}dinger equation from the dynamics of many-body
  quantum systems. On irrational tori, this question was posed as an
  open problem in previous work of Kirkpatrick, Schlein, and
  Staffilani.
\end{abstract}

\thanks{This work has been partially supported by the German Research
  Foundation, CRC 701.}

\maketitle

\section{Introduction}
\subsection{Setup of the problem}
Let $d \geq 2$ be fixed. Suppose that $\theta_1,\theta_2, \ldots,
\theta_d>0$ are given parameters.  We consider the following domain:
\begin{equation}
  \notag
  \Lambda_d=\Lambda_d\,(\theta_1,\theta_2,\ldots,\theta_d):=(\mathbb{R}\, \big/\,\frac{2\pi}{\theta_1}\mathbb{Z}) \times (\mathbb{R} \,\big/\,\frac{2\pi}{\theta_2} \mathbb{Z}) \times \cdots \times (\mathbb{R} \,\big/\,\frac{2\pi}{\theta_d} \mathbb{Z}).
\end{equation}
For the purpose of this paper we call $\Lambda_d$ a \emph{general
  (rectangular d-dimensional) torus}.

In the context of nonlinear dispersive equations, general rectangular
tori were first studied in the work of Bourgain \cite{B07}, where it
was noted that the number-theoretic methods employed in the case of
the classical $d$-dimensional torus $\mathbb{T}^d=(\mathbb{R}\,
\big/\, 2\pi \mathbb{Z})^d$, i.e.\
$\theta_1=\theta_2=\cdots=\theta_d=1$, cannot be used in general.  As
a result, proving dispersive estimates is more challenging on general
tori.

In this work, we will consider the \emph{Gross-Pitaevskii hierarchy on
  the spatial domain $\Lambda_d$}. We recall that this is a system of
infinitely many equations, which is given by:
\begin{equation}
  \label{GP1}
  \begin{cases}
    i \partial_t \gamma^{(k)}+(\Delta_{\vec{x}_k}-\Delta_{\vec{x}'_k}) \gamma^{(k)}=b_0 \cdot \sum_{j=1}^{k}B_{j,k+1}(\gamma^{(k+1)})\\
    \gamma^{(k)} |_{t=0}=\gamma^{(k)}_0.
  \end{cases}
\end{equation}
We use similar notation as in \cite{GSS,VS2}: For fixed $k \in
\mathbb{N}$, $\gamma_0^{(k)}$ is a complex-valued function on
$\Lambda_d^k \times \Lambda_d^k$. Such a function is in general called
\emph{a density matrix of order $k$ on $\Lambda_d$}. Furthermore, each
$\gamma^{(k)}=\gamma^{(k)}(t)$ is a time-dependent density matrix of
order $k$ on $\Lambda_d$. If we denote by
$(\vec{x}_k;\vec{x}'_k)=(x_1,\ldots,x_k;x'_1,\ldots,x'_k)$ the spatial
variable of $\Lambda_d^k \times \Lambda_d^k$, then
$\Delta_{\vec{x}_k}:=\sum_{j=1}^{k} \Delta_{x_j}$ is the Laplace
operator in the first component of $\Lambda_d^k$ and
$\Delta_{\vec{x}'_k}:=\sum_{j=1}^{k} \Delta_{x'_j}$ is the Laplace
operator in the second component of $\Lambda_d^k$. Given $j \in
\{1,2,\ldots,k\}$, $B_{j,k+1}$ denotes the \emph{collision operator},
which is defined as $B_{j,k+1}(\sigma^{(k+1)}):=\,Tr_{x_{k+1}}
\big[\delta(x_j-x_{k+1}),\sigma^{(k+1)}\big]$, whenever $\sigma^{(k+1)}$ is a density matrix of order $k+1$ on
$\Lambda_d$. Here, $\delta$ denotes the Dirac delta function and
$Tr_{x_{k+1}}$ denotes the trace in the $x_{k+1}$ variable, which we sometimes also denote as $Tr_{k+1}$, for simplicity of notation.  A more
detailed definition of the collision operator is given in
\eqref{Bjk+1} below. Finally, $b_0 \in \mathbb{R}$ is a non-zero
coupling constant. If $b_0>0$, we say that the problem \eqref{GP1} is
\emph{defocusing}, and if $b_0<0$, we say that it is \emph{focusing}.

The problem \eqref{GP1} is closely linked to the \emph{cubic nonlinear
  Schr\"{o}dinger equation (NLS) on $\Lambda_d$}:
\begin{equation}
  \label{NLS}
  \begin{cases}
    i \partial_t \phi_t + \Delta \phi_t = b_0 \cdot |\phi_t|^2 \phi_t\\
    \phi_t \big|_{t=0}=\phi
  \end{cases}
\end{equation} 
The coupling constant $b_0$ is the same as in \eqref{GP1}.  From a
solution $\phi_t$ to \eqref{NLS} we can construct a solution of
\eqref{GP1} with initial data $\gamma_0^{(k)}=|\phi \rangle \langle
\phi|^{\otimes k}$. This is the solution given by:
\begin{equation*}
  \gamma^{(k)}(t):=|\phi_t \rangle \langle \phi_t|^{\otimes k}.
\end{equation*} 
Here, $|\cdot \rangle \langle \cdot|$ denotes the Dirac bracket, which
is given by $|h \rangle \langle g|(x;x'):=h(x) \cdot
\overline{g(x')}.$ The sequence $(\gamma^{(k)}(t))_k$ is then called a
\emph{factorized solution of \eqref{GP1}}.

The factorized solutions of \eqref{GP1} play a key role in the
rigorous derivation of the NLS from the dynamics of many-body quantum
systems. More precisely, let us first start from a real-valued
potential $V$, which is defined on $\Lambda_d$. Given $N \in
\mathbb{N}$, we build the corresponding $N$-body Hamiltonian $H_N$ on
a dense subspace of $L^2_{sym}(\Lambda_d^N)$, which is the space of
all permutation-symmetric elements of $L^2(\Lambda^N)$. The operator
$H_N$ is given by:
\[
H_N:=-\sum_{j=1}^{N}\Delta_{x_j}+\frac{1}{N}\sum_{\ell<j}^{N}V_N(x_{\ell}-x_j).
\]
Here $V_N(x):=N^{3\beta}V(N^{\beta}x)$, for $\beta>0$ a
parameter. Given initial data $\Psi_{N,0}$, we can solve the $N$-body
Schr\"{o}dinger equation associated to $H_N$:
\begin{equation*}
  \begin{cases}
    i \partial_t \Psi_{N,t}=H_N \Psi_{N,t} \\
    \Psi_{N,t}\big|_{t=0}=\Psi_{N,0}.
  \end{cases}
\end{equation*}
From the solution, we can define:
\begin{equation}
  \label{gammakNt}
  \gamma^{(k)}_{N,t}:=Tr_{k+1,\ldots,N} \,\big|\Psi_{N,t} \rangle \langle \Psi_{N,t}\big|.
\end{equation}
Here, $Tr_{k+1,\ldots,N}$ denotes the partial trace in the
$x_{k+1},\ldots,x_N$ variables. By definition, for $k>N$, one takes
$\gamma^{(k)}_{N,t}:=0$.

The constructed sequence $(\gamma^{(k)}_{N,t})_k$ solves the
\emph{Bogoliubov-Born-Green-Kirkwood-Yvon (BBGKY) hierarchy on
  $\Lambda_d$}:
\begin{align*}
  & i \partial_t \gamma^{(k)}_{N,t} + \big(\Delta_{\vec{x}_k}-\Delta_{\vec{x}'_k}\big) \gamma^{(k)}_{N,t}\\
  =&\frac{1}{N} \sum_{\ell<j}^{k}
  \big[V_{N}(x_{\ell}-x_j),\gamma^{(k)}_{N,t}\big] +
  \frac{N-k}{N}\sum_{j=1}^{k} Tr_{k+1}
  \big[V_{N}(x_j-x_{k+1}),\gamma^{(k+1)}_{N,t}\big].
\end{align*}

Formally speaking, the BBGKY hierarchy converges to the GP hierarchy
with $b_0=\int_{\Lambda_d} V(x)\,dx$ as $N \rightarrow \infty$. In
order to make this formal argument rigorous, one wants to choose
$(\Psi_N)_N$ in an appropriate way in terms of $\phi$ and show that
there exists a sequence $N_j \rightarrow \infty$, which does not
depend on $k \in \mathbb{N}$ and $t$ with the property that:
\begin{equation}
  \label{convergence}
  Tr\,\Big|\gamma^{(k)}_{N_j,t}-|\phi_t \rangle \langle \phi_t|^{\otimes k} \Big| \rightarrow 0
\end{equation}
as $j \rightarrow \infty$, for $t$ belonging to a finite time
interval. Here, $Tr \big| \cdot \big|$ denotes the trace norm. We
refer to \eqref{convergence} as \emph{a rigorous derivation of the
  cubic NLS on $\Lambda_d$ from the dynamics of many-body quantum
  systems}.

\subsection{Statement of the results}
In our paper, we will primarily study the case $d=2$ and $d=3$,
i.e. the setting of two- and three-dimensional general rectangular
tori.  The following class of time-dependent density matrices is
defined on $\Lambda_d$:
\begin{definition}
  \label{mathcalAtilde}
  Given $\alpha \in \mathbb{R}$, let
  $\widetilde{\mathcal{A}}=\widetilde{A}(\alpha)$ denote the class of
  all time-dependent sequences
  $\widetilde{\Gamma}(t)=(\widetilde{\gamma}^{(k)}(t))$, where each
  $\widetilde{\gamma}^{(k)}: \mathbb{R}_t \times \Lambda_d^k \times
  \Lambda_d^k \rightarrow \mathbb{C}$ satisfies:
  \begin{itemize}
  \item[i)] For all $t \in \mathbb{R},x_1,\ldots,x_k,x'_1, \ldots,x'_k
    \in \Lambda_d$, and for all $\sigma \in
    S^k$: \[\widetilde{\gamma}^{(k)}(t,x_{\sigma(1)},\ldots,x_{\sigma(k)};x_{\sigma(1)}',\ldots,x_{\sigma(k)}')=\widetilde{\gamma}^{(k)}(t,x_1,\ldots,x_k;x_1',\ldots,x_k').\]
  \item[ii)] There exist positive and continuous functions
    $\widetilde{f},\widetilde{g} : \mathbb{R} \rightarrow \mathbb{R}$,
    which are independent of $k$, such that for all $t \in \mathbb{R}$
    and for all $j \in \{1,2,\ldots,k\}$:
    \begin{equation*}
      \int_{t-\widetilde{g}(t)}^{t+\widetilde{g}(t)} \|S^{(k,\alpha)}B_{j,k+1}(\widetilde{\gamma}^{(k+1)})(s)\|_{L^2(\Lambda_d^k \times \Lambda_d^k)} ds \leq \widetilde{f}^{\,k+1}(t).
    \end{equation*}
  \end{itemize}
\end{definition}
Here $S^{(k,\alpha)}$ denotes the operator of fractional
differentiation of order $\alpha$ on density matrices of order $k$, as
is defined in \eqref{Skalpha} below. The class
$\widetilde{\mathcal{A}}$ corresponds to the a priori bound needed in
the analysis of \cite{KM}. The time dependence of the parameters
$\widetilde{f}(t)$ and $\widetilde{g}(t)$ was subsequently introduced
in \cite{GSS}.

In the two-dimensional case, we will prove the following conditional
uniqueness result:
\begin{maintheorem}\label{thm:main1}
  Let $\alpha>\frac{1}{2}$ be given. Then, solutions to the
  Gross-Pitaevskii hierarchy on $\Lambda_2$ are unique in the class
  $\widetilde{\mathcal{A}}(\alpha)$.
\end{maintheorem}

Theorem \ref{thm:main1} resolves an open problem for the irrational
torus stated in \cite{KSS}. In particular, we can implement the result
obtained in Theorem \ref{thm:main1} into the derivation strategy and combine
the analysis of \cite{KSS} adapted to the case of general
two-dimensional tori to deduce:
\begin{maintheorem}\label{thm:main2}
  The convergence in \eqref{convergence} holds on $\Lambda_2$ for $V
  \in W^{2,\infty}(\Lambda_2)$ with $V \geq 0$, $\int_{\Lambda_2}
  V(x)\,dx=b_0>0$, for $\beta \in (0,\frac{3}{4})$, and for
  $(\psi_N)_N \in \mathop{\bigoplus}_{N} L^2(\Lambda_2^N)$ satisfying
  the assumptions of:
  \begin{itemize}
  \item[i)] Bounded energy per particle:
    \begin{equation}
      \label{Bounded energy per particle}
      \sup_{N \in \mathbb{N}} \frac{1}{N} \langle \psi_N, H_N \psi_N \rangle < \infty.
    \end{equation}
  \item[ii)] Asymptotic factorization, i.e.\ there exists $\phi \in
    H^1(\Lambda_2)$ with $\|\phi\|_{L^2(\Lambda_2)}=1$ such that:
    \begin{equation}
      \label{Asymptotic factorization}
      Tr \big|\gamma_N^{(1)}-|\phi \rangle \langle \phi|^{\otimes k} \big| \rightarrow 0\,\,\,\mbox{as}\,\,N \rightarrow \infty.
    \end{equation}
  \end{itemize}
  The function $\phi$ is taken as the initial data in \eqref{NLS}.
\end{maintheorem}

Theorem \ref{thm:main1} is stated as Theorem
\ref{uniqueness_Lambda2_2D} and Theorem \ref{thm:main2} is stated as
Theorem \ref{NLS_Lambda2} below.

In the three-dimensional setting, we can prove the following
conditional uniqueness result:
\begin{maintheorem}\label{thm:main3}
  Let $\alpha>1$ be given. Then, solutions to the Gross-Pitaevskii
  hierarchy on $\Lambda_3$ are unique in the class
  $\widetilde{\mathcal{A}}(\alpha)$. Moreover, whenever $\alpha \geq
  1$, the class $\widetilde{\mathcal{A}}(\alpha)$ is non-empty and it
  contains the factorized solutions corresponding to initial data in
  $H^{\alpha}(\Lambda_3)$.
\end{maintheorem}

As in \cite{GSS}, the uniqueness in Theorem \ref{thm:main3} is above
the regularity of the energy space and hence is not enough to obtain a
rigorous derivation result. However, it is possible to use the
multilinear estimates on the irrational torus \cite{Strunk} --see also
the recent preprint \cite{Killip_Visan}-- and adapt the arguments in
\cite{VS2} to $\Lambda_3$ to obtain the following unconditional
uniqueness result:
\begin{maintheorem}\label{thm:main4}
  Suppose that $\Gamma(t)=(\gamma^{(k)}(t))_k \in L^{\infty}_{t \in
    [0,T]} \mathfrak{H}^1$ is a mild solution to the Gross-Pitaevskii
  hierarchy on $\Lambda_3$ such that each component $\gamma^{(k)}(t)$
  can be written as a limit in the weak-$*$ topology of the trace
  class on $L^2_{sym}(\Lambda_3^k \times \Lambda_3^k)$ of
  $Tr_{k+1,\ldots,N} \Gamma_{N,t}$. Here, each $\Gamma_{N,t} \in
  L^2_{sum}(\Lambda_3^N \times \Lambda_3^N)$ is non-negative as an
  operator and it has trace equal to $1$. Then $\Gamma(t)$ is uniquely
  determined by the initial data $\Gamma(0)$.
\end{maintheorem}

Theorem \ref{thm:main4} can hence be used to deduce the following
derivation result:
\begin{maintheorem}\label{thm:main5}
  Let $V: \mathbb{R}^3 \rightarrow \mathbb{R}$ be non-negative,
  smooth, compactly supported and suppose that
  $b_0=\int_{\mathbb{R}^3} V(x)\,dx=b_0>0$. Moreover, let $\beta \in
  (0,\frac{3}{5})$, and let $(\psi_N)_N \in \mathop{\bigoplus}_{N}
  L^2(\Lambda_3^N)$ satisfy the assumptions \eqref{Bounded energy per
    particle} and \eqref{Asymptotic factorization}.
\end{maintheorem}
Then, the convergence \eqref{convergence} holds.  Theorem
\ref{thm:main3}, Theorem \ref{thm:main4}, and Theorem \ref{thm:main5}
are stated respectively as Theorem \ref{uniqueness_Lambda3_3D},
Theorem \ref{unconditional_uniqueness_3D}, and Theorem
\ref{NLS_Lambda3} below.  A key idea in the proof of Theorem
\ref{thm:main1} and Theorem \ref{thm:main3} is to prove certain
spacetime estimates. The spacetime estimate in two dimensions is given
in Proposition \ref{2D_spacetime_estimate} and sharpness is shown in
Proposition \ref{Lower_bound_2D}. In three dimensions, the spacetime
estimate is given in Proposition \ref{3D_spacetime_estimate} and
sharpness is shown in Proposition \ref{Lower_bound_3D}. The sharp
spacetime estimate in the general case of $d$ dimensions is noted in
Remark \ref{higher_dimensions}. Once we have a good spacetime
estimate, it is possible to apply the analysis in the work of T.~Chen
and Pavlovi\'{c} \cite{CP4} to obtain local-in-time solutions to the
hierarchy. The details of this construction are recalled in Section
\ref{Local-in-time solutions}.

\subsection{Physical interpretation}
Finally, we briefly recall the physical interpretation, similarly to
\cite[Subsection 1.3.2]{GSS}: The Gross-Pitaevskii hierarchy and the
nonlinear Schr\"{o}dinger equation occur in the framework of
Bose-Einstein condensation. This is a state of matter which is made up
of bosonic particles which are cooled to a temperature close to
absolute zero. In such an environment, they tend to occupy the lowest
quantum state. This state corresponds to a ground state of an energy
functional associated to an NLS-type equation. In such a context, the
NLS equation is sometimes refered to as the \emph{Gross-Pitaevskii
  equation}, after the work of Gross \cite{Gross} and Pitaevskii
\cite{Pitaevskii}.  The phenomenon of Bose-Einstein condensation was
theoretically predicted in 1924-1925 in the work of Bose \cite{Bose}
and Einstein \cite{Einstein}. The existence of this state of matter
was experimentally verified independently by the research teams of
Cornell and Wieman \cite{CW} and Ketterle \cite{Ketterle} in
1995. Both experimental groups were awarded the Nobel Prize in Physics
in 2001 for their achievement.

\subsection{Previous results}

In this subsection, we will briefly discuss further related
results and provide references, which in part is similar to the
exposition in \cite{GSS,VS2}.

Properties of the ground state for the $N$-body Hamiltonian $H_N$ have
been studied by Lieb and Seiringer \cite{LS}, Lieb, Seiringer and
Yngvason \cite{LSY,LSY2}, and Lieb, Seiringer, Yngvason, and Solovej
\cite{LSSY2}. In these works, the assumption of \emph{asymptotic
  factorization} as in \eqref{Asymptotic factorization} of the initial
data was rigorously verified for a sequence consisting of appropriate
ground states. These results are summarized in the expository work
\cite{LSSY}.

A method for proving \eqref{convergence} by means of the BBGKY hierarchy was first developed by Spohn
\cite{Spohn} when the spatial domain is $\mathbb{R}^d$. Using this approach, the result obtained in \cite{Spohn}
is a rigorous derivation of the Hartree equation $i u_t+\Delta
u=(V*|u|^2)u$ on $\mathbb{R}^d$ with $V \in L^{\infty}(\mathbb{R}^d)$. Due to the regularity assumptions on the convolution potential, in this case it is possible to explicitly compute the infinite Duhamel expansions $\gamma^{(k)}_{\infty,t}$ and show that it is the appropriate limit of $(\gamma^{(k)}_{N,t})_N$. This result could be extended to the situation of more singular $V$ in \cite{BEGMY,BGM,EY} by adapting the following two-step derivation strategy.  In the first step, it is shown that the
sequence $(\gamma^{(k)}_{N,t})_N$ satisfies certain compactness
properties, and that the obtained subsequential limits solve the GP
hierarchy.  In the second step, it is shown that solutions to the GP
hierarchy are unique in a class of objects containing this limit. Due
to the complexity of the system, the uniqueness step is significantly
non-trivial.

A breakthrough in that direction was the rigorous derivation of
the defocusing cubic NLS on $\mathbb{R}^3$ in the work of Erd\H{o}s,
Schlein, and Yau \cite{ESY2,ESY3,ESY4,ESY5}. Here, the uniqueness step
was obtained by the use of Feynman graph
expansions.  A combinatorial reformulation of this step under an
additional a priori assumption on the solution was subsequently given
by Klainerman and Machedon \cite{KM}.

An alternative rigorous derivation of
NLS-type equations based on Fock space
techniques was developed by Hepp \cite{Hepp} and Ginibre and
Velo \cite{GV1,GV2}.


The first derivation result in the periodic problem was given in the
case of $\mathbb{T}^2$ in the work of Kirkpatrick, Schlein, and
Staffilani \cite{KSS}. We note that, in the case of $\mathbb{T}^3$,
the first step in the derivation strategy was carried out in the work of
Elgart, Erd\H{o}s, Schlein, and Yau \cite{EESY}, which builds on the
work previously done in \cite{ESY1}. A conditional uniqueness result
for the GP hierarchy on $\mathbb{T}^3$ was shown in a class of density
matrices of regularity $\alpha>1$ by Gressman, the second author, and
Staffilani in \cite{GSS}. This was done by the use of a spacetime
estimate as in \cite{KSS,KM}. In \cite{GSS}, it was shown that the
obtained range of regularity exponents in this estimate was
sharp. Since the regularity is above the natural energy space, which
corresponds to the regularity $\alpha=1$, it is not immediately
possible to apply this result in the second step of the derivation.

In a recent paper, T.~Chen, Hainzl, Pavlovi\'{c}, and Seiringer
\cite{ChHaPavSei} give an alternative proof of the uniqueness result
on $\mathbb{R}^3$ by use of the \emph{(Weak) Quantum de Finetti
  Theorem}, formulated in
\cite{AmmariNier1,AmmariNier2,LewinNamRougerie}. This approach was
subsequently adapted to the setting of $\mathbb{T}^3$ by the second
author in \cite{VS2}, in which the open question of uniqueness from
\cite{EESY} was resolved. As a result, one could obtain a rigorous
derivation of the defocusing cubic NLS on $\mathbb{T}^3$. We note that
the uniqueness result of \cite{VS2} does not directly extend the
uniqueness result of \cite{GSS} since the papers deal with different
classes of density matrices.
Methods based on the Quantum de Finetti Theorem were applied to
related problems in \cite{ChHaPavSei2,CS,HTX}.

Moreover, once one has a rigorous
derivation result, based on either method, it is possible to study the
rate of convergence in \eqref{convergence}. This problem was first
addressed by Rodnianski and Schlein \cite{RodnianskiSchlein}, and subsequently reformulated in \cite{KP}.
The Cauchy problem associated to the GP hierarchy was studied in its
own right in the work of T.~Chen and Pavlovi\'{c} \cite{CP1}, with
later work in \cite{CP2,CP4,CP3,CPT1,CPT2,LewinSabin1,LewinSabin2}.
Randomization techniques were studied in
the context of the Cauchy problem associated to the GP hierarchy in
\cite{VS1,SoSt}. In a recent preprint \cite{LewinNamRougerie4}, a
derivation of the nonlinear Gibbs measure from many-body dynamics is
given. A related result in a discrete setting had also been proved in
\cite{Knowles_Thesis}.
The Klainerman-Machedon a priori bound was studied in further detail
in \cite{CP,CT,ChenHolmer2,ChenHolmer5}. Generalizations of the spacetime
estimate for the GP hierarchy were studied in a different context in
\cite{Beckner1,Beckner2}, with related work in \cite{XC3}. The case of singular convolution potentials was revisited in \cite{FKS,FrTsYau3}. For a more detailed discussion on all of
these results and further references, we refer the reader to \cite[Subsection 1.3.2]{GSS} and \cite[Section 1.1 and Section 1.3]{VS2}, as well as to the expository works \cite{CP_survey_article,Schlein}.

As was noted above, the study of the NLS on general rectangular tori
was first started in the work of Bourgain \cite{B07}. Here, it was
shown that certain Strichartz estimates with a loss of derivative
hold. The estimates were much weaker than those proved in the setting
of the classical torus \cite{B}, due to number-theoretical
difficulties. The NLS on irrational tori has been studied further in
\cite{Catoire_Wang,Demirbas,Guo_Oh_Wang,Strunk}. A different approach
to the problem, which has led to stronger results, including the
stronger Strichartz estimates conjectured in \cite{B07}, was recently
taken in the work of Bourgain and Demeter
\cite{Bourgain_Demeter1,Bourgain_Demeter2,Bourgain_Demeter3}, with
subsequent work by Killip and Vi\c{s}an \cite{Killip_Visan}.

\subsection{Main ideas of the proof}
We choose the setup of \cite{B07,Guo_Oh_Wang,Strunk}. To this end, we will rescale the domain
$\Lambda_d$ to the classical torus $\Lambda$. Consequently, we will be
able to work with Fourier series defined on the standard lattice
$\mathbb{Z}^d$. Due to the scaling transformation, we will have to
work with the modified Laplacian $\Delta_Q$, as well as with the
associated quadratic form $Q$. In the rescaled setting, we will
consider a Gross-Pitaevskii hierarchy with modified Laplacian on
$\Lambda$ as in \eqref{GP}.

A crucial ingredient in proving conditional uniqueness for this
hierarchy on $\Lambda$ is a spacetime estimate associated to
$\,\mathcal{U}^{(k)}_Q(t)$, the free Schr\"{o}dinger evolution on
density matrices of order $k$ on $\Lambda$ corresponding to
$\Delta_Q$. The operator $\mathcal{U}^{(k)}_Q(t)$ is precisely defined
in \eqref{Uk_Q}. This estimate is stated in two dimensions in
Proposition \ref{2D_spacetime_estimate} and in three dimensions in
Proposition \ref{3D_spacetime_estimate}. By standard techniques, the
spacetime bound is reduced to a pointwise estimate on a corresponding
multiplier $I(\tau,p)$ as in \cite{GSS}.

A challenge in proving the pointwise estimate lies in the fact that
the sum in the multiplier is taken over a larger set than in the
classical setting. In particular, the sum does not contain a
$\delta$-function, but it contains a characteristic function of the
interval $[0,1]$. Geometrically speaking, we are no longer summing
over lattice points that lie on a curve, but over lattice points that
lie near a curve. This is a general phenomenon which occurs, due to
the fact that $e^{it\Delta_Q}$ is in general no longer periodic in
time. To this sum, we apply the dyadic decomposition arguments from
\cite{GSS}. Due to the presence of the quadratic form $Q$, associated
to $\Delta_Q$, which can take irrational values, the arguments based
on the determinant of a lattice from \cite{GSS} do not apply in this
setting. We will remedy this difficulty by using a Fourier analytic
fact used in \cite{B07}, which allows us to estimate the number of
points in a set in terms of an integral
\eqref{E_tau_p_Fourier_transform_bound}. We then rewrite the integral
in the upper bound in a factorized way \eqref{E_tau_p}, keeping in
mind all of the dyadic localization of the frequencies. In certain
cases, this allows us to apply the oscillatory sum estimates given in
Subsection \ref{Estimates for oscillatory sums} below. The remaining
cases are then obtained from this one by applying the geometric
arguments from \cite{GSS}.

In proving the sharpness of the Sobolev exponents in the spacetime
estimates, we again encounter the difficulty that $e^{it\Delta_Q}$ is
no longer periodic in time.  More precisely, we need to estimate the
integral in $\tau$, the frequency variable corresponding to $t$. In
other words, we need to estimate the behavior of the spacetime Fourier
transform of the expression left-hand side of the estimate for $\tau$
taking values in an interval, rather than at a single point. In order
to prove this fact, we observe that the spacetime Fourier transform
can bounded from below by an integral of a constant function. In the
two-dimensional problem, this is seen in \eqref{tau1_integral}
below. This idea generalizes to all dimensions $d \geq 2$.

Once one proves the spacetime estimate, it is possible to use standard
arguments to deduce a conditional uniqueness result for the
Gross-Pitaevskii hierarchy with a modified Laplacian on $\Lambda$. An
additional scaling argument then allows us to show conditional
uniqueness for the Gross-Pitaevskii hierarchy on $\Lambda_d$.  We note
that, by this method, it is not immediately possible to deduce
conditional uniqueness for the Gross-Pitaevskii hierarchy on
$\Lambda_d$ from the spacetime estimate on $\Lambda$. This point is
explained in more detail in Remark
\ref{2D_spacetime_estimate_bound_Lambda2_remark} below.

\subsection{Organization of the paper} In Section \ref{Notation and
  some preliminary facts}, we will define the notation and we will
recall some useful preliminary facts from Fourier analysis. In
Subsection \ref{Density matrices and their properties}, we will define
more precisely the notions related to density matrices. In Subsection
\ref{Estimates for oscillatory sums}, we will recall several estimates
for oscillatory sums.

Section \ref{2D_problem} is devoted to the study of the
two-dimensional problem.  In Subsection \ref{The spacetime estimate in
  two dimensions}, we will prove a sharp spacetime estimate for the
free evolution operator $\mathcal{U}^{(k)}_Q(t)$ in two dimensions. A
conditional uniqueness result for the GP hierarchy on $\Lambda_2$ is
given in Subsection \ref{A conditional uniqueness result in 2D}. This
result will be used to obtain a rigorous derivation of the defocusing
cubic NLS on $\Lambda_2$ in Subsection \ref{Rigorous derivation 2D
  irrational torus}.

The three-dimensional problem is studied in Section
\ref{3D_problem}. In Subsection \ref{The spacetime estimate in three
  dimensions}, we prove a sharp spacetime estimate for
$\mathcal{U}_Q^{(k)}(t)$. This result is used in order to deduce a
conditional uniqueness result for the GP hierarchy on $\Lambda_3$ in
Subsection \ref{A conditional uniqueness result 3D}. The sharp
spacetime estimate for general dimensions is given in Remark
\ref{higher_dimensions}. An unconditional uniqueness result, which is
used to obtain a rigorous derivation of the defocusing cubic NLS on
$\Lambda_3$ is given in Subsection \ref{An unconditional uniqueness
  result 3D}.

In Section \ref{Local-in-time solutions}, we comment on the existence
of local-in-time solutions to the GP hierarchy on general $\Lambda_d$.

\section{Notation and some preliminary facts}
\label{Notation and some preliminary facts}

\subsection{Density matrices and their properties}
\label{Density matrices and their properties}
Given $f \in L^2(\Lambda_d)$, its Fourier transform is defined as
follows:
\begin{equation}
  \label{Fourier_transform_Lambda_d}
  \widehat{f}(\xi):=\int_{0}^{\frac{2 \pi}{\theta_1}} \int_{0}^{\frac{2 \pi}{\theta_2}} \cdots \int_{0}^{\frac{2 \pi}{\theta_d}} f(x^1,x^2,\ldots,x^d) \cdot e^{-ix^1 \cdot \xi^1 - ix^2 \cdot \xi^2 - \cdots - ix^d \cdot \xi^d} \,dx^d\,\cdots\,dx^2\,dx^1.
\end{equation}
Here, $\xi=(\xi^1,\xi^2,\ldots,\xi^d) \in \theta_1 \cdot \mathbb{Z}
\times \theta_2 \cdot \mathbb{Z} \times \cdots \times \theta_d \cdot
\mathbb{Z}$.  On $\Lambda_d$, we consider the standard Laplace
operator given by:
$$\Delta=\frac{\partial^2}{\partial (x^1)^2}+\frac{\partial^2}{\partial (x^2)^2}+\cdots+\frac{\partial^2}{\partial (x^d)^2}.$$
Here, $x^j$ denotes the $\mathbb{R}\,\big/\,\frac{2\pi}{\theta_j}
\mathbb{Z}$ variable.  In particular, given a function $f \in
L^2(\Lambda_d)$ such that $\Delta f \in L^2(\Lambda_d)$, and
$\xi=(\xi^1,\xi^2,\ldots,\xi^d) \in \theta_1 \cdot \mathbb{Z} \times
\theta_2 \cdot \mathbb{Z} \times \cdots \times \theta_d \cdot
\mathbb{Z}$ as above, it is the case that:
\begin{equation}
  \notag
  (\Delta f)\,\widehat{\,}\,(\xi)=(-(\xi^1)^2-(\xi^2)^2-\cdots-(\xi^d)^2) \cdot \widehat{f}(\xi),
\end{equation}
When working with density matrices of order $k$, we will use the
shorthand notation $\vec{x}_k:=(x_1,x_2,\ldots,x_k)$,
$\vec{x}'_k:=(x'_1,x'_2,\ldots,x'_k)$, where
$x_j=((x_j)^1,(x_j)^2,\ldots,(x_j)^d)\in \Lambda_d$,
$x'_j=((x'_j)^1,(x'_j)^2,\ldots,(x'_j)^d) \in \Lambda_d$.

Given $\gamma^{(k)}$, a density matrix of order $k$ on $\Lambda_d$,
its Fourier transform is defined as:

\begin{equation}
  \label{gamma_k_hat}
  (\gamma^{(k)})\,\widehat{\,}\,(\vec{\xi}_1;\vec{\xi}'_k):=\int_{\Lambda_d^k \times \Lambda_d^k} \gamma^{(k)}(\vec{x}_k;\vec{x}'_k)e^{-i \cdot \sum_{j=1}^{k}x_j \cdot \xi_j + i \cdot \sum_{j=1}^{k} x'_j \cdot \xi'_j}\,d\vec{x}_k\,d\vec{x}'_k
\end{equation}
for $\vec{\xi}_k=(\xi_1,\ldots,\xi_k),
\vec{\xi}'_k:=(\xi'_1,\ldots,\xi'_k) \in (\theta_1 \cdot \mathbb{Z}
\times \theta_2 \cdot \mathbb{Z} \times \cdots \times \theta_d \cdot
\mathbb{Z})^k$.
Furthermore, if $\gamma^{(k)}=\gamma^{(k)}(t)$ depends on time, we can
define its spacetime Fourier transform as:
\begin{equation}
  \label{gamma_k_tilde}
  (\gamma^{(k)})\,\widetilde{\,}\,(\tau,\vec{\xi}_1;\vec{\xi}'_k):=\int_{\mathbb{R}} \int_{\Lambda_d^k \times \Lambda_d^k} \gamma^{(k)}(t,\vec{x}_k;\vec{x}'_k)e^{-it\tau-i \cdot \sum_{j=1}^{k}x_j \cdot \xi_j + i \cdot \sum_{j=1}^{k} x'_j \cdot \xi'_j}\,d\vec{x}_k\,d\vec{x}'_k\,dt,
\end{equation}
for $\tau \in \mathbb{R}$ and $\vec{\xi}_k,\vec{\xi}'_k \in (\theta_1
\cdot \mathbb{Z} \times \theta_2 \cdot \mathbb{Z} \times \cdots \times
\theta_d \cdot \mathbb{Z})^k$.

The \emph{collision operator} $B_{j,k+1}$, for $j \in
\{1,2,\ldots,k\}$ is defined on density matrices of order $k+1$ by:
\begin{equation}
  \label{Bjk+1}
  B_{j,k+1} \gamma^{(k+1)}(\vec{x}_k;\vec{x}_k'):= Tr_{x_{k+1}}\, \big[\delta(x_j-x_{k+1}),\gamma^{(k+1)}\big](\vec{x}_k;\vec{x}_k')
\end{equation}
$$=\int_{\Lambda_d} dx_{k+1} \Big(\delta(x_j-x_{k+1})-\delta(x_j'-x_{k+1})\Big) \gamma^{(k+1)}(\vec{x}_k,x_{k+1};\vec{x}_k',x_{k+1}).
$$
In particular, $B_{j,k+1}\gamma^{(k+1)}$ is a density matrix of order
$k$. We can write the above difference as $B_{j,k+1}^{+}
\gamma^{(k+1)}-B_{j,k+1}^{-}\gamma^{(k+1)}$. In this way, we define
the operators $B_{j,k+1}^{+}$ and $B_{j,k+1}^{-}$. More precisely,
\begin{equation}
  \label{B+}
  B_{j,k+1}^{+} \gamma^{(k+1)}(\vec{x}_k;\vec{x}_k'):=\int_{\Lambda_d} dx_{k+1} \,\delta(x_j-x_{k+1}) \gamma^{(k+1)}(\vec{x}_k,x_{k+1};\vec{x}_k',x_{k+1})
\end{equation}
and
\begin{equation}
  \label{B-}
  B_{j,k+1}^{-} \gamma^{(k+1)}(\vec{x}_k;\vec{x}_k'):=\int_{\Lambda_d} dx_{k+1} \,\delta(x_j'-x_{k+1}) \gamma^{(k+1)}(\vec{x}_k,x_{k+1};\vec{x}_k',x_{k+1}).
\end{equation}
Given $\alpha \in \mathbb{R}$, we define the \emph{differentiation
  operator} $S^{(k,\alpha)}$ on density matrices of order $k$ on
$\Lambda_d$ by:
\begin{equation}
  \label{Skalpha}
  (S^{(k,\alpha)}\gamma^{(k)})\,\widehat{\,}\,\,(\vec{\xi}_k;\vec{\xi'}_k):=\prod_{j=1}^{k} \langle \xi_j \rangle^{\alpha} \cdot \prod_{j=1}^{k} \langle \xi'_j \rangle^{\alpha} \cdot\widehat{\gamma}^{(k)}(\vec{\xi}_k;\vec{\xi'}_k),
\end{equation}
whenever $\vec{\xi}_k, \vec{\xi}'_k \in (\theta_1 \cdot \mathbb{Z}
\times \theta_2 \cdot \mathbb{Z} \times \cdots \times \theta_d \cdot
\mathbb{Z})^k$. Here, we use the convention for the Japanese bracket
given by:
\begin{equation}
  \notag
  \langle x \rangle:=\sqrt{1+|x|^2}. 
\end{equation}

Using $\Delta$, we define $\mathcal{U}^{(k)}(t)$ by:
\begin{equation*}
  \mathcal{U}^{(k)}(t) \,\gamma^{(k)}:=e^{it \sum_{j=1}^{k} \Delta_{x_j}} \gamma^{(k)} e^{-it\sum_{j=1}^{k} \Delta_{x'_j}}.
\end{equation*}
This operator corresponds to the \emph{free Schr\"{o}dinger evolution}
on density matrices.

In the remainder of this section, we will rescale the torus
$\Lambda_d$ to the classical torus $\Lambda=\mathbb{T}^d=(\mathbb{R}
\,\big/\,2\pi\mathbb{Z})^d$. We will henceforth work on the simpler
domain $\Lambda$ at the expense of working with a modified Laplacian
operator as it was done in the context of the NLS in \cite{B07}.  More
precisely, we will consider the domain
$\Lambda=\mathbb{T}^d=(\mathbb{R}/2\pi\mathbb{Z})^d$ with the modified
Laplacian operator:
\begin{equation}
  \notag
  \Delta_Q=\theta_1^2 \cdot \frac{\partial^2}{\partial (x^1)^2}+\theta_2^2 \cdot \frac{\partial^2}{\partial (x^2)^2} + \cdots + \theta_d^2 \cdot \frac{\partial^2}{\partial (x_d)^2}.
\end{equation}
Here, we write the spatial variable as $x=(x^1,x^2,\ldots,x^d) \in
\mathbb{T}^d$.

The action of $\Delta_Q$ on the Fourier side is given by:

\begin{equation}
  \label{Delta_Q}
  (\Delta_Q g)\,\widehat{\,}\,(\xi):=(-\theta_1^2 \cdot (\xi^1)^2 -\theta_2^2 \cdot (\xi^2)^2-\cdots-\theta_d^2 \cdot (\xi_d)^2) \cdot \widehat{g}(\xi)
\end{equation}
for $\xi=(\xi^1,\xi^2,\ldots,\xi^d) \in \mathbb{Z}^d$.  In
\eqref{Delta_Q}, $\,\widehat{\cdot}\,$ denotes the Fourier transform
on $\Lambda$, i.e.
\begin{equation}
  \label{Fourier_transform_Lambda}
  \widehat{g}(\xi)=\int_{0}^{2\pi} \int_{0}^{2\pi} \cdots \int_{0}^{2\pi} f(x^1,x^2,\ldots,x^d) \cdot e^{-ix^1 \cdot \xi^1 - ix^2 \cdot \xi^2 - \cdots - ix^d \cdot \xi^d} \,dx^d\,\cdots\,dx^2\,dx^1,
\end{equation}
whenever $g \in L^2(\Lambda)$. Let us note that the Fourier transform
on $\Lambda_d$, given in \eqref{Fourier_transform_Lambda_d}, and the
Fourier transform on $\Lambda$, given in
\eqref{Fourier_transform_Lambda}, are denoted in the same way. From
context, it will be possible to distinguish which Fourier transform we
are considering.

Similarly, the analogues on $\Lambda$ of the operations given in
\eqref{gamma_k_hat}, \eqref{gamma_k_tilde}, \eqref{Bjk+1}, and
\eqref{Skalpha} are defined by simply setting
$\theta_1=\theta_2=\cdots=\theta_d=1$ in the definitions given
above. In what follows, we will denote the operations of Fourier
transform, spacetime Fourier transform, collision, and fractional
differentiation acting on density matrices on $\Lambda_d$ and on
$\Lambda$ in the same way. It will typically be clear from context
which operation we are using.

We associate to $\Delta_Q$ the operator $\mathcal{U}_Q^{(k)}(t)$,
which acts on density matrices of order $k$ on $\Lambda$ by:

\begin{equation}
  \label{Uk_Q}
  \mathcal{U}_Q^{(k)}(t) \, \gamma^{(k)}:=e^{it\sum_{j=1}^{k}\Delta_{Q,x_j}}\, \gamma^{(k)} \,e^{-it \sum_{j=1}^{k} \Delta_{Q,x'_j}}.
\end{equation}
Here, $\Delta_{Q,x_j}$ denotes the operator $\Delta_Q$ acting in the
$x_j$ variable and $\Delta_{Q,x'_j}$ denotes the operator $\Delta_Q$
acting in the $x'_j$ variable. Let us note that, unlike
$\,\mathcal{U}^{(k)}(t)$, \emph{the operator $\mathcal{U}^{(k)}_Q(t)$
  is in general not periodic in time}.

We will now define a method of rescaling density matrices. Given $k
\in \mathbb{N}$ and $\gamma^{(k)}:\Lambda^k \times \Lambda^k
\rightarrow \mathbb{C}$, we define
$\widetilde{\gamma}^{(k)}:\Lambda_d^k \times \Lambda_d^k \rightarrow
\mathbb{C}$ as:
\begin{equation}
  \label{rescaling}
  \widetilde{\gamma}^{(k)}(x_1,x_2,\ldots,x_k;x'_1,x'_2,\ldots,x'_k):=\gamma^{(k)}(\widetilde{x}_1,\widetilde{x}_2,\ldots,\widetilde{x}_k;\widetilde{x}'_1,\widetilde{x}'_2,\ldots,\widetilde{x}'_k)
\end{equation}
where, for $j \in \mathbb{N}$, given $x_j, x'_j \in \Lambda_d$, we
define:
\begin{equation}
  \notag
  \widetilde{x}_j:=\big(\theta_1 \cdot (x_j)^1,\theta_2 \cdot (x_j)^2, \cdots, \theta_d \cdot (x_j)^d \big) \in \Lambda
\end{equation}
and
\begin{equation}
  \notag
  \widetilde{x}'_j:=\big(\theta_1 \cdot (x'_j)^1,\theta_2 \cdot (x'_j)^2, \cdots, \theta_d \cdot (x'_j)^d \big) \in \Lambda.
\end{equation}
The mapping defined in \eqref{rescaling} gives a bijection between
density matrices of order $k$ on $\Lambda$ and on $\Lambda_d$.  By the
Chain Rule, it follows that, with the above rescaling:
\begin{equation}
  \label{Delta_Delta_Q1}
  \Big((\Delta_{\vec{x}_k}-\Delta_{\vec{x}'_k})\widetilde{\gamma}^{(k)}\Big)(\vec{x}_k;\vec{x}'_k)\\
  = \Big((\Delta_{Q,\vec{x}_k}-\Delta_{Q,\vec{x}'_k})\gamma^{(k)}\Big)(\widetilde{x}_1,\ldots,\widetilde{x}_k;\widetilde{x}'_1,\ldots,\widetilde{x}'_k).
\end{equation}
Here, $\Delta_{Q,\vec{x}_k}:=\sum_{j=1}^{k}\Delta_{Q,x_j}$ and
$\Delta_{Q,\vec{x}'_k}:=\sum_{j=1}^{k}\Delta_{Q,x'_j}$.

We can take Fourier transforms of both sides of \eqref{rescaling} and
obtain that, for all $\vec{\xi}_k=(\xi_1,\ldots,\xi_k),
\vec{\xi}'_k:=(\xi'_1,\ldots,\xi'_k) \in (\theta_1 \cdot \mathbb{Z}
\times \theta_2 \cdot \mathbb{Z} \times \cdots \times \theta_d \cdot
\mathbb{Z})^k$:
\begin{equation}
  \label{Fourier_transform_gamma_tilde}
  \widehat{\widetilde{\gamma}^{(k)}}
  (\vec{x}_k;\vec{x}'_k)=\frac{1}{\theta_1^{2k} \cdot \theta_2^{2k} \cdots \theta_d^{2k}} \cdot  \widehat{\gamma^{(k)}}(\widetilde{\xi}_1,\widetilde{\xi}_2,\ldots,\widetilde{\xi}_k;\widetilde{\xi}'_1,\widetilde{\xi}'_2,\ldots,\widetilde{\xi}'_k).
\end{equation}
Here,
\begin{equation}
  \label{rescaling_frequency1}
  \widetilde{\xi}_j:=\Big(\frac{1}{\theta_1} \cdot (\xi_j)^1,\frac{1}{\theta_2} \cdot (\xi_j)^2,\cdots, \frac{1}{\theta_d} \cdot (\xi_j)^d\Big) \in \mathbb{Z}^d
\end{equation}
and
\begin{equation}
  \label{rescaling_frequency2}
  \widetilde{\xi}'_j:=\Big(\frac{1}{\theta_1} \cdot (\xi'_j)^1,\frac{1}{\theta_2} \cdot (\xi'_j)^2, \cdots, \frac{1}{\theta_d} \cdot (\xi'_j)^d \Big) \in \mathbb{Z}^d.
\end{equation}
Here, and throughout the paper, we use the notational convention
\[\xi_j=\big((\xi_j)^1,(\xi_j)^2, \ldots, (\xi_j)^d\big) \text{ and }
\xi'_j=\big((\xi'_j)^1,(\xi'_j)^2, \ldots, (\xi'_j)^d \big).\]

In addition to \eqref{GP1}, we will also study the
\emph{Gross-Pitaevskii hierarchy on the spatial domain $\Lambda$ with
  a modified Laplacian:}
\begin{equation}
  \label{GP}
  \begin{cases}
    i \partial_t \gamma^{(k)}+(\Delta_{Q,\vec{x}_k}-\Delta_{Q,\vec{x}'_k}) \gamma^{(k)}=b_0 \cdot \sum_{j=1}^{k}B_{j,k+1}(\gamma^{(k+1)})\\
    \gamma^{(k)}|_{t=0}=\gamma^{(k)}_0.
  \end{cases}
\end{equation}
Here, $\gamma^{(k)} : \mathbb{R}_t \times \Lambda^k \times \Lambda^k
\rightarrow \mathbb{C}$ and $\gamma^{(k)}_0 : \Lambda^k \times
\Lambda^k \rightarrow \mathbb{C}$. As above, $b_0 \in \mathbb{R}$ is a
non-zero coupling constant. As in \eqref{GP1}, we say that \eqref{GP}
is \emph{defocusing} if $b_0>0$, and we say that it is \emph{focusing}
if $b_0<0$.

Suppose that $(\gamma^{(k)})_k=(\gamma^{(k)}(t))_k$ solves
\eqref{GP}. For each $k \in \mathbb{N}$ and for each time $t$, we
define $\widetilde{\gamma}^{(k)}(t)$ from $\gamma^{(k)}(t)$ according
to \eqref{rescaling},
i.e. \[\widetilde{\gamma}^{(k)}(t,\vec{x}_k;\vec{x}'_k):=\gamma^{(k)}(t,\widetilde{x}_1,\ldots,\widetilde{x}_k;\widetilde{x}'_1,\ldots,\widetilde{x}'_k).\]
A direct calculation then shows that, for all $k \in \mathbb{N}$, and
for all $j \in \{1,2,\ldots,k\}$:
\begin{equation}
  \label{Bjk_gamma_gamma_tilde}
  \big[B_{j,k+1} (\gamma^{(k+1)})\big](t,\vec{x}_k;\vec{x}'_k)=
  \big[B_{j,k+1} (\widetilde{\gamma}^{(k+1)})\big](t,\widetilde{x}_1,\ldots,\widetilde{x}_k;\widetilde{x}'_1,\ldots,\widetilde{x}'_k).
\end{equation}
From \eqref{Delta_Delta_Q1} and \eqref{Bjk_gamma_gamma_tilde}, we
obtain:

\begin{lemma}
  \label{correspondence}
  $(\widetilde{\gamma}^{(k)})_k$ solves \eqref{GP1} if and only if
  $(\gamma^{(k)})_k$ solves \eqref{GP}.
\end{lemma}

Hence, by using the rescaling \eqref{rescaling}, we obtain a
\emph{correspondence between solutions of \eqref{GP1} and \eqref{GP}.}

Given $\xi=(\xi^1,\xi^2, \ldots, \xi^d),
\eta=(\eta^1,\eta^2,\ldots,\eta^d) \in \mathbb{Z}^d$, we define
\begin{equation}
  \notag
  Q(\xi,\eta):=\theta_1^2 \cdot \xi^1 \cdot \eta^1 + \theta_2^2 \cdot \xi^2 \cdot \eta^2 + \cdots + \theta_d^2 \cdot \xi^d \cdot \eta^d.
\end{equation}

In the sequel, we will use the abbreviated notation:
\begin{equation}
  \notag
  Q(\xi):=Q(\xi,\xi).
\end{equation}
In other words, we can write \eqref{Delta_Q} as:

\begin{equation}
  \notag
  (\Delta_Q f)\,\widehat{\,}\,(\xi)=-Q(\xi) \cdot \widehat{f}(\xi)
\end{equation}
for $\xi \in \mathbb{Z}^d$.  Moreover, if
$\vec{\xi}_k=(\xi_1,\ldots,\xi_k) \in (\mathbb{Z}^{d})^k$, we define
$Q(\vec{\xi}_k):=\sum_{j=1}^{k}Q(\xi_j)$.

In the paper, we will primarily focus on the case when $d=2$ and
$d=3$.  The two-dimensional problem will be considered in Section
\ref{2D_problem} and the three-dimensional problem will be considered
in Section \ref{3D_problem}.  A statement for general $d \geq 2$ is
given in Remark \ref{higher_dimensions}, Remark
\ref{higher_dimensions_uniqueness}, and in Section \ref{Local-in-time
  solutions} below.  Moreover, we will primarily be interested in the
case when $\Lambda_d$ is a irrational torus in the sense that it is
not possible to use a rescaling argument and apply the results from
the setting of the classical torus $\T^d$ obtained in \cite{GSS,KSS}.

\subsection{Estimates for oscillatory sums}
\label{Estimates for oscillatory sums}
Let us first recall a simple and well-known $L^4$-bound, cp.~\cite[Lemma 3.2]{BGT05}, as well as \cite{B07,H13} for related results:
\begin{lemma}
  \label{lem:l4}
  Let $I\subset \mathbb{R}$ be a bounded interval and let
  $\epsilon>0$. There exists $c>0$, depending on $\epsilon$, such that
  for all $b \in \mathbb{Z}$ and $N\in \mathbb{N}$,
  \[
  \int_I \Big| \sum_{m \in [b,b+N)\cap \mathbb{Z}}e^{it
    m^2}\Big|^4dt\leq c N^{2+\epsilon}
  \]
\end{lemma}
\begin{proof}
  Without loss it suffices to consider $I=[0,2\pi]$. Then, by
  Plancherel,
  \begin{align*}
    L:=&\int_I \Big| \sum_{m \in [b,b+N)\cap \mathbb{Z}}e^{it
      m^2}\Big|^4dt=\Big\| \sum_{m_1,m_2 \in [b,b+N)\cap
      \mathbb{Z}}e^{it
      (m_1^2-m_2^2)}\Big\|_{L^2([0,2\pi])}^2\\
    =&\sum_{l} \Big|\sum_{m_1,m_2 \in [b,b+N)\cap \mathbb{Z} \atop
      m_1^2-m_2^2=l} 1\Big|^2
  \end{align*}
  Since there are at most $N^2$ different values for $l$ such that
  $m_1^2-m_2^2=l$ has a solution $(m_1,m_2)$ in the given interval, we
  obtain
  \[
  L\leq N^2 \sup_l S_{l,b}(N)^2,\] where
  \[S_{l,b}(N)=\#\{(m_1,m_2)\in [b,b+N)^2:(m_1-m_2)(m_1+m_2)=l \}.
  \]
  Hence, it remains to give a uniform estimate for
  $S_{l,b}(N)$. Setting $k_1:=m_1-m_2$, $k_2:=m_1+m_2-2b$, we observe
  that
  \[
  S_{l,b}(N) =\#\{(k_1,k_2)\in [-N,N)\times[0,2N):k_1(k_2+2b)=l \}.
  \]
  We discuss two cases separately.

  {\it Case 1:} $-10N^2\leq b \leq 10N^2$. Then, if solutions exist,
  we must have $l \leq cN^3$, and the classical estimate on the
  number-of-divisors function $d$ \cite[Thm.~315]{HW54} yields
  \[
  S_{l,b}(N)\leq d(l)\leq c l^\delta\leq c N^{3\delta},
  \]
  for any $\delta>0$.

  {\it Case 2:} $|b|>10N^2$. In this case, for any given $l$, there is
  at most one solution to $k_1(k_2+2b)=l$. Indeed, suppose that
  $(k_1,k_2)$ and $(k_1',k_2')$ are two distinct solutions, hence
  $k_1\ne k_1'$. Then,
  \[
  k_1k_2-k_1'k_2'=2b(k_1'-k_1) \Rightarrow |k_1k_2-k_1'k_2'|>20N^2,
  \]
  a contradiction to $|k_1k_2-k_1'k_2'|\leq 4N^2$.
\end{proof}

\begin{corollary}
  \label{cor:lp} Let $p \geq 4$, $\epsilon>0$ and $I\subset
  \mathbb{R}$ be a bounded interval.  There exists $c>0$, depending on
  $\epsilon$, such that for all $b \in \mathbb{Z}$ and $N\in
  \mathbb{N}$,
  \[
  \int_I \Big| \sum_{m \in [b,b+N)\cap \mathbb{Z}}e^{it m^2}\Big|^p
  dt\leq c N^{p-2+\epsilon}.
  \]
\end{corollary}

\begin{proof}
  Let us note that, for all $b \in \mathbb{R}$ and $N \in \mathbb{N}$,
  it is obvious that
  \begin{equation*}
    \Big\|\sum_{m \in [b,b+N) \cap \mathbb{Z}} e^{itm^2}\Big\|_{L^\infty_t(I)} \leq N.
  \end{equation*}
  The corollary now follows from Lemma \ref{lem:l4} and interpolation.
\end{proof}
We remark that in the case $p>4$ the result of Corollary \ref{cor:lp}
holds with $\epsilon=0$, see \cite{B07,BGT05,H13}, but we do not need
this here.

\section{The two-dimensional problem}
\label{2D_problem}

In this section, we consider the two-dimensional problem. In other
words, we fix $d=2$ in the notation given above. Throughout the
section, $\Lambda$ will denote the two-dimensional classical torus
$\mathbb{T}^2$.  In Subsection \ref{The spacetime estimate in two
  dimensions}, we will prove a sharp spacetime estimate for the free
evolution $\mathcal{U}_Q^{(k)}(t)$ associated to the classical torus
$\Lambda=\mathbb{T}^2$. We will use this result and a scaling argument
to prove a conditional uniqueness result on $\Lambda_2$ in Subsection
\ref{A conditional uniqueness result in 2D}. Finally, in Subsection
\ref{Rigorous derivation 2D irrational torus}, we will use the
conditional uniqueness result in order to obtain a rigorous derivation
of the defocusing cubic NLS on the two-dimensional general rectangular
torus, as was done in the setting of the two-dimensional classical
torus in \cite{KSS}.

\subsection{The spacetime estimate in two dimensions}
\label{The spacetime estimate in two dimensions}

The following spacetime estimate, which extends a result of
\cite{KSS}, will be key in our analysis:
\begin{proposition}
  \label{2D_spacetime_estimate}
  Let $\alpha>\frac{1}{2}$ be given. There exists $C>0$, which depends
  only on $\alpha, \theta_1, \theta_2$ such that, for all $k \in
  \mathbb{N}$, and for all $\gamma_0^{(k+1)}: \Lambda^{k+1} \times
  \Lambda^{k+1} \rightarrow \mathbb{C}$, the following estimate holds:
  \begin{equation}
    \label{2D_spacetime_estimate_bound}
    \|S^{(k,\alpha)} B_{j,k+1} \,\mathcal{U}_Q^{(k+1)}(t)\,\gamma_0^{(k+1)}\|_{L^2([0,1] \times \Lambda^k \times \Lambda^k)} \leq C \|S^{(k+1,\alpha)}\gamma_0^{(k+1)}\|_{L^2(\Lambda^{k+1} \times \Lambda^{k+1})}.
  \end{equation}
\end{proposition}

The range of regularity exponents $\alpha>\frac{1}{2}$ in Proposition
\ref{2D_spacetime_estimate} is sharp due to the following:

\begin{proposition}
  \label{Lower_bound_2D}
  For $\kappa \in \mathbb{N}$ sufficiently large, there exists
  $\gamma^{(2)}_0: \Lambda^2 \times \Lambda^2 \rightarrow \mathbb{C}$,
  such that for $\delta>0$ sufficiently small:
  \begin{equation}
    \notag
    \|S^{(1,\frac{1}{2})}B_{1,2}\,\mathcal{U}_Q^{(2)}(t)\,\gamma^{(2)}_0\|_{L^2([0,\delta] \times \Lambda \times \Lambda)} \gtrsim_{\,\delta} \sqrt{\ln \kappa} \cdot \|S^{(2,\frac{1}{2})} \gamma^{(2)}_0\|_{L^2(\Lambda^2 \times \Lambda^2)}.
  \end{equation}
  Here, $\delta$ is independent of $\kappa$.
\end{proposition}

Let us first prove Proposition \ref{2D_spacetime_estimate}.
\begin{proof}[Proof of Proposition \ref{2D_spacetime_estimate}]

  Let $\psi \in \mathcal{S}(\mathbb{R})$ be such that $\psi \geq 1$ on
  $[0,1]$. It suffices to show that:
  \begin{equation}
    \notag
    \|\psi(t)S^{(k,\alpha)}B_{1,k+1}^{+}\,\mathcal{U}_Q^{(k+1)}(t)\gamma_0^{(k+1)}\|_{L^2(\mathbb{R} \times \Lambda^k \times \Lambda^k)}
  \end{equation}
  is bounded by the expression on the right-hand side in
  \eqref{2D_spacetime_estimate_bound}. Here, the operator
  $B_{j,k+1}^{\pm}$ is defined as in \eqref{B+} and \eqref{B-} when we
  set $\theta_1=\theta_2=1$.  The estimate when $B_{1,k+1}^{+}$ is
  replaced by general $B_{j,k+1}^{\pm}$ is proved in the same way.

  We compute
  \begin{align*}
    &\big(\psi(t) S^{(k,\alpha)}B_{1,k+1}^{+}\,\mathcal{U}_Q^{(k+1)}(t)\,\gamma_0^{(k+1)}\big)\,\widetilde{\,}\,(\tau,\vec{\xi}_k;\vec{\xi'}_k)\\
    =&\sum_{\xi_{k+1},\xi'_{k+1} \in \mathbb{Z}^2} \widehat{\psi}\big(\tau+Q(\xi_1-\xi_{k+1}+\xi'_{k+1})+Q(\vec{\xi}_{k+1})-Q(\xi_1)-Q(\vec{\xi}'_{k+1})\big) \cdot \\
    &\qquad \cdot \prod_{j=1}^{k+1} \langle \xi_j \rangle^{\alpha}
    \cdot \prod_{j=1}^{k+1} \langle \xi'_j \rangle^{\alpha} \cdot
    \widehat{\gamma}_0^{(k+1)}(\xi_1-\xi_{k+1}+\xi'_{k+1},\xi_2,\ldots,\xi_{k+1};\xi'_1,\ldots,\xi'_{k+1}).
  \end{align*}
  Hence, by the Cauchy-Schwarz inequality:
  \[
  \Big|\big(\psi(t)
  S^{(k,\alpha)}B_{1,k+1}^{+}\,\mathcal{U}_Q^{(k+1)}(t)\,\gamma_0^{(k+1)}\big)\,\widetilde{\,}\,(\tau,\vec{\xi}_k;\vec{\xi'}_k)\Big|
  \leq \big(\Sigma_1\big)^{\frac{1}{2}} \cdot
  \big(\Sigma_2\big)^{\frac{1}{2}}\] where
  \begin{align*}
    &\Sigma_1\\
    :=&\sum_{\xi_{k+1},\xi'_{k+1} \in \mathbb{Z}^2}
    \frac{\big|\widehat{\psi}\big(\tau+Q(\xi_1-\xi_{k+1}+\xi'_{k+1})+Q(\vec{\xi}_{k+1})-Q(\xi_1)-Q(\vec{\xi}'_{k+1})\big)\big|^2
      \cdot \langle \xi_1\rangle^{2\alpha}}{\langle
      \xi_1-\xi_{k+1}+\xi'_{k+1} \rangle ^{2\alpha} \cdot \langle
      \xi_{k+1} \rangle^{2\alpha} \langle \xi'_{k+1}
      \rangle^{2\alpha}}
  \end{align*}
  and
  \begin{align*}
    \Sigma_2:=&\sum_{\xi_{k+1},\xi'_{k+1} \in \mathbb{Z}^2} \langle \xi_1-\xi_{k+1}+\xi'_{k+1}\rangle^{2\alpha} \cdot \prod_{j=2}^{k+1} \langle \xi_j \rangle^{2\alpha} \cdot \prod_{j=1}^{k+1} \langle \xi'_j \rangle^{2\alpha} \\
    &\cdot
    |\widehat{\gamma}_0^{(k+1)}(\xi_1-\xi_{k+1}+\xi'_{k+1},\xi_2,\ldots,\xi_{k+1};\xi'_1,\ldots,\xi'_{k+1})|^2.
  \end{align*}
  In particular, the claim will follow if we show that $\Sigma_1$ is
  uniformly bounded in $(\tau,\vec{\xi}_k;\vec{\xi'}_k)$.  Let us
  analyze $\Sigma_1$ more closely. We note that it can be decomposed
  as $ \Sigma_1=\sum_{\ell \in \mathbb{Z}}\Sigma_1(\ell)$ for
  \begin{align*}
    &\Sigma_1(\ell)\\
    :=& \sum_{\xi_{k+1},\xi'_{k+1} \in \mathbb{Z}^2 : \ast}
    \frac{\big|\widehat{\psi}\big(\tau+Q(\xi_1-\xi_{k+1}+\xi'_{k+1})+Q(\vec{\xi}_{k+1})-Q(\xi_1)-Q(\vec{\xi}'_{k+1})\big)\big|^2
      \langle \xi_1\rangle^{2\alpha}}{\langle
      \xi_1-\xi_{k+1}+\xi'_{k+1} \rangle ^{2\alpha} \langle \xi_{k+1}
      \rangle^{2\alpha} \langle \xi'_{k+1} \rangle^{2\alpha}},
  \end{align*}
  where we sum with respect to the constraint $\ast$ given by
  \[
  \tau+Q(\xi_1-\xi_{k+1}+\xi'_{k+1})+Q(\vec{\xi}_{k+1})-Q(\xi_1)-Q(\vec{\xi}'_{k+1})
  \in [\ell,\ell+1].
  \]
  We can choose $\psi$ such that $\widehat{\psi}$ is increasing on
  $(-\infty,0)$ and decreasing on $(0,+\infty)$ (for example, we can
  take $\psi(t)=ce^{-t^2}$ for an appropriate choice of $c>0$). Then,
  \begin{align*}
    &\Sigma_1(\ell)\\
    \lesssim & \sum_{\ell \in \mathbb{Z}} |\widehat{\psi}(\ell)|^2
    \cdot \sup_{m \in \mathbb{Z}} \Big[\sum_{\xi_{k+1},\xi'_{k+1} \in
      \mathbb{Z}^2: \ast \ast }\frac{\langle \xi_1
      \rangle^{2\alpha}}{\langle \xi_1-\xi_{k+1}+\xi'_{k+1}
      \rangle^{2\alpha} \cdot \langle \xi_{k+1} \rangle^{2\alpha}
      \cdot \langle \xi'_{k+1} \rangle^{2\alpha}}\Big]
  \end{align*}
  with respect to the constraint $\ast \ast$ gived by
  \[
  \tau+Q(\xi_1-\xi_{k+1}+\xi'_{k+1})+Q(\vec{\xi}_{k+1})-Q(\xi_1)-Q(\vec{\xi}'_{k+1})
  \in [m,m+1].\] Hence,
  \begin{align*}
    &\Sigma_1(\ell)\\
    \lesssim & \sup_{m \in \mathbb{Z}} \sum_{\xi_{k+1},\xi'_{k+1}\in
      \mathbb{Z}^2}
    \frac{\widetilde{\delta}\big(\tau-m+Q(\xi_1-\xi_{k+1}+\xi'_{k+1})+Q(\vec{\xi}_{k+1})-Q(\xi_1)-Q(\vec{\xi}'_{k+1})\big)\cdot
      \langle \xi_1 \rangle^{2\alpha}}{\langle
      \xi_1-\xi_{k+1}+\xi'_{k+1} \rangle^{2\alpha} \cdot \langle
      \xi_{k+1} \rangle^{2\alpha} \cdot \langle \xi'_{k+1}
      \rangle^{2\alpha}}.
  \end{align*}
  Here, $\widetilde{\delta}:=\chi_{[0,1]}$ is the characteristic
  function of the interval $[0,1]$.  We rewrite
  \begin{align*}
    &\tau+Q(\xi_1-\xi_{k+1}+\xi'_{k+1})+Q(\vec{\xi}_{k+1})-Q(\xi_1)-Q(\vec{\xi}'_{k+1})\\
    =&\big(\tau+\sum_{\ell=2}^{k}
    Q(\xi_{\ell})-\sum_{\ell=1}^{k}Q(\xi'_{\ell})\big)+Q(\xi_1-\xi_{k+1}+\xi'_{k+1})+Q(\xi_{k+1})-Q(\xi'_{k+1}).
  \end{align*}
  After changing variables $\tau \mapsto
  \tau+\sum_{\ell=2}^{k}Q(\xi_{\ell})-\sum_{\ell=1}^{k}Q(\xi'_{\ell})$,
  and letting $p:=\xi_1,n:=\xi_{k+1},m:=-\xi_{k+1}$, it follows that
  we need to bound:
  \begin{equation}
    \label{I_tau_p1}
    I(\tau,p):=\sum_{m,n \in \mathbb{Z}^2} \frac{\widetilde{\delta}\big(\tau+Q(p-n-m)+Q(n)-Q(m)\big) \cdot \langle p \rangle^{2\alpha}}{\langle p-n-m \rangle^{2\alpha} \cdot \langle n \rangle^{2\alpha} \cdot \langle m \rangle^{2\alpha}}
  \end{equation}
  uniformly in $\tau \in \mathbb{R}$ and $p \in \mathbb{Z}^2$.  By
  using the same calculations as in \cite[equation (48)]{GSS}, we can
  rewrite $I(\tau,p)$ as
  \begin{equation}
    \label{I_tau_p2}
    I(\tau,p)=\sum_{m,n \in \mathbb{Z}^2} \frac{\widetilde{\delta}\big(\tau+Q(p)-2Q(n,m)\big) \cdot \langle p \rangle^{2\alpha}}{\langle m-p \rangle^{2\alpha} \cdot \langle n-p \rangle^{2\alpha} \cdot \langle p-n-m \rangle^{2\alpha}}.
  \end{equation}
  Obviously, the contribution of $m=0$ or $n=0$ is uniformly bounded if $\alpha>\frac12$, so we can restrict the sum to $\mathbb{Z}^2\setminus\{0\} \times
      \mathbb{Z}^2\setminus\{0\}$.
  The advantage of writing $I(\tau,p)$ as in \eqref{I_tau_p2} instead
  of as in \eqref{I_tau_p1} is that, in \eqref{I_tau_p2}, for a fixed
  $m \neq 0$, one sums over the set of all $n$ which lie in a
  neighborhood of a fixed hyperplane.

  Given $j=(j_1,j_2,j_3) \in \mathbb{N}_0^3$, $\tau \in \mathbb{R}$,
  and $p \in \mathbb{Z}^2$, we let $E_{\tau,p}(j) \subseteq \mathbb{Z}^2\setminus\{0\} \times
    \mathbb{Z}^2\setminus\{0\}$
  be the set of all pairs $(m,n)$ such that
  \begin{equation}
    \notag
    \begin{cases}
      \tau + Q(p)-2Q(n,m) \in [0,1]\\
      |m-p| \sim 2^{j_1}, |n-p| \sim 2^{j_2}, |p-n-m| \sim 2^{j_3}.
    \end{cases}
  \end{equation}
  Here, by $|x| \sim 2^{j}$, we mean $2^{j-1} \leq |x| < 2^j$ when $j
  \geq 1$ and $|x| < 1$ when $j=0$.  In what follows, we will order
  $j_1,j_2,j_3$ as $j_{\min} \leq j_{\med} \leq j_{\max}$.  We would
  like to prove that
  \begin{equation}
    \label{E_tau_p_bound_2D}
    \#E_{\tau,p}(j) \lesssim_{\epsilon} 2^{(1+\epsilon)j_{\min}+(1+\epsilon)j_{\med}}
  \end{equation}
  for all $\epsilon>0$. The calculations in \cite[equations (54) and
  (55)]{GSS} then imply the claim.

  Let us fix $\tau,p,j$ as above. We will now estimate $\#
  E_{\tau,p}(j)$. We first argue by using an idea from \cite{B07}. In
  particular, let us fix $\phi \in C_0^{\infty}(\mathbb{R})$ such that
  $\widehat{\phi} \geq 0$ on all of $\mathbb{R}$ and $\widehat{\phi}
  \geq 1$ on $[0,1]$, see Lemma \ref{Fourier_Analysis_Lemma}.

  By using the same argument to deduce \cite[formula (1.1.8')]{B07},
  it follows that:
  \begin{equation}
    \label{E_tau_p_Fourier_transform_bound}
    \#E_{\tau,p}(j) \leq \int \Big[\mathop{\sum_{n,m \in \mathbb{Z}^2}}_{|m-p| \sim 2^{j_1}, |n-p| \sim 2^{j_2}, |p-m-n| \sim 2^{j_3}}e^{2iQ(n,m)t}\Big] \cdot e^{-i(\tau+Q(p))t} \cdot \phi(t)\,dt.
  \end{equation}
  This is the case since:
  \begin{equation}
    \notag
    \int e^{2iQ(n,m)t} \cdot e^{-i(\tau+Q(p))t} \phi(t)\, dt=\widehat{\phi}\big(\tau+Q(p)-2Q(m,n)\big).
  \end{equation}

  \begin{remark}
    \label{E_tau_p_mn_set}
    Let us note that the estimate in
    \eqref{E_tau_p_Fourier_transform_bound} holds if we replace the
    sum in $m$ and $n$ by the sum over a larger set in $m$ and
    $n$. This follows from the fact that $\widehat{\phi} \geq 0$ on
    all of $\mathbb{R}$. We will use this observation several times in
    the discussion that follows.
  \end{remark}
  We note that:
  \begin{equation}
    \notag
    2Q(n,m)=\frac{Q(n+m)-Q(n-m)}{2}.
  \end{equation}
  So, the right-hand side of \eqref{E_tau_p_Fourier_transform_bound}
  equals:
  \begin{equation}
    \label{E_tau_p_RHS}
    \int \Big[\mathop{\sum_{n,m \in \mathbb{Z}^2}}_{|m-p| \sim 2^{j_1}, |n-p| \sim 2^{j_2}, |p-m-n| \sim 2^{j_3}}e^{\frac{1}{2} it \big(Q(n+m)-Q(n-m)\big)}\Big] \cdot e^{-i(\tau+Q(p))t} \cdot \phi(t)\,dt.
  \end{equation}
  Let $\eta:=n-m$ and $\eta':=n+m$. This is a one-to-one change of
  variables. We would like to find the localization properties of the
  $\eta$ and $\eta'$ variables thus defined.

  We first note that:

\begin{equation}
  \notag
  \eta=n-m=
\end{equation}
\begin{equation}
  \notag
  =\underbrace{(n-p)}_{2^{j_2}}-\underbrace{(m-p)}_{2^{j_1}}=\underbrace{p-m-n}_{2^{j_3}}+\underbrace{2n-2p}_{2^{j_2+1}}+p=\underbrace{-p+n+m}_{2^{j_3}}-\underbrace{(2m-2p)}_{2^{j_1+1}}-p.
\end{equation}
Here, by $\underbrace{x}_{2^j}$, we mean that $|x| \leq 2^j$. Let $a \vee b:=\max\{a,b\}$. Since
\[2^{k_1}+2^{k_2} \leq 2^{k_1\vee k_2+1} \text{ for all }k_1,k_2 \geq 0,\] it
follows that:
\begin{equation}
  \label{eta_localization}
  \eta \in B_{2^{j_1\vee j_2+1}}(0) \cap B_{2^{j_2\vee j_3+2}}(p) \cap B_{2^{j_1\vee j_3+2}}(-p).
\end{equation}
Similarly,
\begin{equation}
  \notag
  \eta'=n+m=\underbrace{(-p+n+m)}_{2^{j_3}}+p=\underbrace{(n-p)}_{2^{j_2}}+\underbrace{(m-p)}_{2^{j_1}}+2p.
\end{equation}
Hence:
\begin{equation}
  \label{eta'_localization}
  \eta' \in B_{2^{j_3}}(p) \cap B_{2^{j_1\vee j_2+1}}(2p).
\end{equation}

We now apply this change of variables in \eqref{E_tau_p_RHS} and
deduce that:
\begin{align}
  &\# E_{\tau,p}(j)\notag\\
  \leq &
  \int \Big[\sum_{\eta,\eta' \in \mathbb{Z}^2 \atop \eqref{eta_localization},\eqref{eta'_localization}} e^{\frac{1}{2} it \big(Q(\eta)-Q(\eta')\big)}\Big] \cdot e^{-i(\tau+Q(p))t} \cdot \phi(t)\,dt\notag\\
  \leq & \int \Big[\sum_{\eta,\eta' \in \mathbb{Z}^2 \atop
    \eqref{eta_localization2},\eqref{eta'_localization2}}
  e^{\frac{1}{2} it \big(Q(\eta)-Q(\eta')\big)}\Big] \cdot
  e^{-i(\tau+Q(p))t} \cdot
  \phi(t)\,dt \label{E_tau_p_Fourier_transform_bound2}
\end{align}
where
\begin{align}
  \eta &\in C_{2^{j_1\vee j_2+1}}(0) \,\cap\, C_{2^{j_2\vee j_3+2}}(p) \,\cap\, C_{2^{j_1\vee j_3+2}}(-p)\label{eta_localization2},\\
  \eta' &\in C_{2^{j_3}}(p) \,\cap\, C_{2^{j_1\vee
      j_2+1}}(2p)\label{eta'_localization2}.
\end{align}
Here $C_{2^j}(q)$ denotes a two-dimensional square, centered at $q$,
whose sides are parallel to the coordinate axes, and who have
sidelength $2^{j+1}$.  Let us note that, in the above calculation, we
used Remark \ref{E_tau_p_mn_set}.  By the triangle inequality and the
fact that $\phi \in C_0^{\infty}(\mathbb{R})$, it follows that:
\begin{equation}
  \label{E_tau_p}
  \#E_{\tau,p}(j) \leq c \cdot \Big\|\sum_{\eta,\eta' \in \mathbb{Z}^2\atop \eqref{eta_localization2},\eqref{eta'_localization2}} e^{\frac{1}{2}it\big(Q(\eta)-Q(\eta')\big)} \Big\|_{L^1_t(I)}
\end{equation}
for some $c>0$ and for some finite interval $I \subseteq \mathbb{R}$.

Let us now use \eqref{E_tau_p} to show \eqref{E_tau_p_bound_2D}. We
will argue by considering all of the possible cases for the relative
sizes of $j_1,j_2,j_3$. In particular, we consider:

\vspace{5mm} \textbf{Case 1:}\,\,$j_3=\min\{j_1,j_2,j_3\}$.
\vspace{5mm}

By \eqref{eta_localization2}, it follows that, in this case, $\eta$ is
localized to a cube of sidelength $O(2^{j_{\med}})$. Moreover, by
\eqref{eta'_localization2}, it follows that $\eta'$ is localized to a
cube of sidelength $O(2^{j_{\min}})$. Hence, in this case, we need to
estimate\footnote{Similarly as in Remark \ref{E_tau_p_mn_set}, it is
  possible to replace the sum in $\eta,\eta'$ in \eqref{E_tau_p} by
  the sum over a larger set. Moreover, for simplicity of notation, we
  write $\eta=(\eta_1,\eta_2), \eta'=(\eta'_1,\eta'_2) \in
  \mathbb{Z}^2$.}:

\begin{equation}
  \label{Case1_Term1}
  \int_{I} \Big| \mathop{\sum_{\eta_1,\eta_2 \in \mathbb{Z}}}_{\eta_j \in I^{\med}_j} \mathop{\sum_{\eta'_1,\eta'_2 \in \mathbb{Z}}}_{\eta'_j \in I^{\min}_j} e^{\,\frac{1}{2}it\big(\theta_1^2\eta_1^2\,+\,\theta_2^2\eta_2^2\,-\,\theta_1^2(\eta'_1)^2\,-\,\theta_2^2(\eta'_2)^2\big)}\Big|dt.
\end{equation}
Here $I^{\med}_1,I^{\med}_2$ are fixed intervals of size $\sim
2^{j_{\med}}$ and $I^{\min}_1,I^{\min}_2$ are fixed intervals of size
$\sim 2^{j_{\min}}$.  By H\"{o}lder's inequality, this expression is:
\begin{equation}
  \notag
  \eqref{Case1_Term1} \leq \Big(\int_{I} \Big|\mathop{\sum_{\eta_1 \in \mathbb{Z}}}_{\eta_1 \in I^{\med}_1} e^{\,\frac{1}{2}it\theta_1^2 \eta_1^2}\Big|^4\,dt\Big)^{\frac{1}{4}} \cdot \Big(\int_{I}\Big|\mathop{\sum_{\eta_2 \in \mathbb{Z}}}_{\eta_2 \in I^{\med}_2} e^{\,\frac{1}{2}it\theta_2^2\eta_2^2}\Big|^4\,dt\Big)^{\frac{1}{4}} \cdot
\end{equation}
\begin{equation}
  \notag
  \cdot \Big(\int_{I} \Big|\mathop{\sum_{\eta'_1 \in \mathbb{Z}}}_{\eta'_1 \in I^{\min}_1} e^{\,\frac{1}{2} it\theta_1^2(\eta'_1)^2}\Big|^4\,dt\Big)^{\frac{1}{4}} \cdot \Big(\int_{I} \Big|\mathop{\sum_{\eta'_2 \in \mathbb{Z}}}_{\eta'_2 \in I^{\min}_2} e^{\,\frac{1}{2} it\theta_2^2 (\eta'_2)^2}\Big|^4\,dt\Big)^{\frac{1}{4}}.
\end{equation}
For each of the four factors, we rescale in time and apply Lemma
\ref{lem:l4} to deduce that this product is:
\begin{equation}
  \notag
  \lesssim_{\epsilon} \big(2^{\,(2+\epsilon)j_{\med}}\big)^{\frac{1}{4}} \cdot \big(2^{\,(2+\epsilon)j_{\med}}\big)^{\frac{1}{4}} \cdot \big(2^{\,(2+\epsilon)j_{\min}}\big)^{\frac{1}{4}} \cdot \big(2^{\,(2+\epsilon)j_{\min}}\big)^{\frac{1}{4}}
\end{equation}
\begin{equation}
  \notag
  \lesssim 2^{\,(1+\epsilon)j_{\min}\,+\,(1+\epsilon)j_{\med}}
\end{equation}
for all $\epsilon>0$. This is a good bound. We note that the implied
constants depend on $\theta_1$ and $\theta_2$.

\vspace{5mm} \textbf{Case 2:}\,\,$j_1=\min\{j_1,j_2,j_3\}$.
\vspace{5mm}

Let us note that, in this case, we can no longer deduce that $\eta'$
is localized to a ball of radius $O(2^{j_{\min}})$. In order to reduce
to the case where one of the variables has a localization of the order
of the smallest frequency, we will apply a covering argument similar
to that which was used in \cite[Proof of Proposition 3.1, Case
3]{GSS}. In particular, given $k=(k_1,k_2) \in \mathbb{Z}^2$, we
consider the rectangle:

\begin{equation}
  \notag
  B_k:=[2^{j_1}k_1,2^{j_1}k_1+2^{j_1}-1] \times [2^{j_1}k_2,2^{j_1}k_2+2^{j_1}-1].
\end{equation}
By using the arguments from Case 1, it follows that for all $(k,k')
\in \mathbb{Z}^2 \times \mathbb{Z}^2$, it is the case that:
\begin{equation}
  \notag
  \#\big(E_{\tau,p}(j) \cap (B_k \times B_{k'})\big) \lesssim_{\epsilon} 2^{\,(2+\epsilon)j_1}
\end{equation}
for all $\epsilon>0$. Namely, in $E_{\tau,p}(j) \cap (B_k \times
B'_k)$, the variables $m$ and $n$, and hence $\eta$ and $\eta'$ are
localized to sets of diameter $O(2^{j_1})$.

Hence, in order to prove \eqref{E_tau_p_bound_2D} in this case, it
suffices to show that:
\begin{equation}
  \notag
  \#\{(k,k') \in \mathbb{Z}^2 \times \mathbb{Z}^2, \,E_{\tau,p}(j) \cap (B_k \times B'_k) \neq \emptyset\} \lesssim 2^{\,\min\{j_2,j_3\}-j_1}.
\end{equation}
Let us recall
\begin{lemma}[Lemma 3.6 from \cite{GSS}]\label{lem:thick}
  Let $Y \subset \mathbb{R}^D$ be any set, and let $Y_{s}$ be the set
  of points in $\mathbb{R}^D$ which are of distance at most $s$ to the
  set $Y$. Let $Z \subset \mathbb{R}^D$ be any $r$-separated set
  (meaning that $|x-x'| \geq r$ whenever $x,x'$ are distinct points in
  $Z$). Then for any $r' > 0$:
$$ \# ( Y_{r'} \cap Z) \lesssim (\min \{r,r'\} )^{-D} |Y_{r'}|.$$
Here, $\big|\cdot\big|$ denotes Lebesgue measure on
$\mathbb{R}^D$. The implied constant depends only on the dimension.
\end{lemma}

We use Lemma \ref{lem:thick} to deduce that:
\begin{equation}
  \notag
  \#\{(k,k') \in \mathbb{Z}^2 \times \mathbb{Z}^2,\,E_{\tau,p}(j) \cap (B_k \times B_{k'}) \neq \emptyset\} \lesssim 2^{-4j_1}\big|X_{2^{j_1}}\big|.
\end{equation}

Here, $X_{2^{j_1}}$ denotes the set of all points in $\mathbb{R}^2
\times \mathbb{R}^2$ which are of distance at most $2^{j_1}$ from the
set $E_{\tau,p}(j)$.  More precisely, $E_{\tau,p}(j) \cap (B_k \times
B_{k'}) \neq \emptyset$ implies that $(k,k')$ belongs to a $2^{j_1}$
thickening of $E_{\tau,p}(j)$.  We then apply Lemma \ref{lem:thick}
with $Z$ being the set of all the centers of $B_k \times B_{k'}$
(which is $2^{j_1}$-separated), with $Y$ being $E_{\tau,p}(j)$, and
with $r=2^{j_1}, r'=2^{j_1}.$ We recall that the dimension $D$ is
equal to $4$.

Thus, we would like to show that:
\begin{equation}
  \notag
  \big|X_{2^{j_1}}\big| \lesssim 2^{\,3j_1\,+\,\min\{j_2,j_3\}}.
\end{equation}

Let $(m,n) \in E_{\tau,p}(j)$.

Then, we know that $m,n \neq 0$ and:
\begin{equation}
  \notag
  \begin{cases}
    |m-p| \sim 2^{j_1},\,|n-p| \sim 2^{j_2},\,|p-n-m| \sim 2^{j_3}\\
    \tau+Q(p)-2Q(n,m) \in [0,1]
  \end{cases}
\end{equation}
We note that \emph{$m$ is allowed to vary over a ball of radius
  $O(2^{j_1})$}. For fixed $m$, the $n$ coordinate is allowed to vary
over a ball of radius $O(2^{\,\min\{j_2,j_3\}})$. Furthermore, since
$\tau+Q(p)-2Q(n,m) \in [0,1]$, it follows that $n$ lies within an
$O(1)$ distance of a fixed line in $\mathbb{R}^2$. Namely, we know
that:

$$\frac{\tau+Q(p)}{2 |m|} -Q\Big(n,\frac{m}{|m|}\Big)= \frac{\tau+Q(p)}{2 |m|}-n_1 \cdot \frac{\theta_1^2 \, m_1}{|m|}-n_2 \cdot \frac{\theta_2^2 \, m_2}{|m|} \in \Big[0,\frac{1}{2|m|}\Big]$$ and $|m| \geq 1$.
Hence, \emph{$n$ lies in the intersection of a ball of radius
  $O(2^{\,\min\{j_2,j_3\}})$ and an $O(1)$ neighborhood of a fixed
  line in $\mathbb{R}^2$.}

Let us now consider the thickening $X_{2^{j_1}}$. If $(x,y) \in
X_{2^{j_1}}$, it follows from the previous arguments that $x$ lies in
a ball of radius $O(2^{j_1})$, and for a fixed $x$, the $y$ coordinate
lies in the intersection of a ball of radius
$O(2^{\,\min\{j_2,j_3\}})$ and an $O(2^{j_1})$ neighborhood of a fixed
line. In deducing the localization for $(x,y)$ from the localization
of $(m,n)$ we used the fact that $j_1=\min\{j_1,j_2,j_3\}$. Let us
note that, in order to deduce the localization properties of the $y$
coordinate, we thicken by an amount of $\sim 2^{j_1}$ in the direction
perpendicular to the line and parallel to the line separately.

Consequently, the Lebesgue measure of the set to which $x$ is
localized is $\lesssim 2^{2j_1}$. For a fixed $x$, the Lebesgue
measure of the set to which $y$ is localized is:
$$\lesssim 2^{j_1} \cdot 2^{\,\min\{j_2,j_3\}}=2^{\,j_1\,+\,\min\{j_2,j_3\}}.$$
We may hence conclude that:
\begin{equation}
  \notag
  \big|X_{2^{j_1}}\big| \lesssim 2^{\,2j_1} \cdot 2^{\,j_1\,+\,\min\{j_2,j_3\}}=2^{\,3j_1\,+\,\min\{j_2,j_3\}}.
\end{equation}
This is the bound that we wanted to show.

\vspace{5mm} \textbf{Case 3:}\,\,$j_2=\min\{j_1,j_2,j_3\}$.
\vspace{5mm}

By symmetry in $m$ and $n$ in the sum defining $I(\tau,p)$, this case
is analogous to Case 2.  Proposition \ref{2D_spacetime_estimate} now
follows.
\end{proof}

\begin{remark}
  \label{2D_spacetime_estimate_bound_Lambda2_remark}
  If we apply the rescaling \eqref{rescaling} in inequality
  \eqref{2D_spacetime_estimate_bound}, we can deduce a spacetime
  estimate on $\Lambda_2$. More precisely, we can deduce that, given
  $\alpha>\frac{1}{2}$, there exists $C_1>0$ depending on $\alpha,
  \theta_1, \theta_2$ such that for all $k \in \mathbb{N}$ and for all
  density matrices $\gamma_0^{(k)}$ of order $k$ on $\Lambda_d$, it is
  the case that:
  \begin{equation}
    \label{2D_spacetime_estimate_bound_Lambda2}
    \|S^{(k,\alpha)} B_{j,k+1} \,\mathcal{U}^{(k+1)}(t)\,\gamma_0^{(k+1)}\|_{L^2([0,1] \times \Lambda_2^k \times \Lambda_2^k)} \leq C_1^k \|S^{(k+1,\alpha)}\gamma_0^{(k+1)}\|_{L^2(\Lambda_2^{k+1} \times \Lambda_2^{k+1})}.
  \end{equation}
  The reason why we obtain a $k$-th power of $C_1$ is that the Sobolev
  norms which we are using are inhomogeneous. Due to the additional
  $k$ dependence in the constant, it is not possible to directly use
  \eqref{2D_spacetime_estimate_bound_Lambda2} and prove a conditional
  uniqueness result for \eqref{GP1}. We will circumvent this
  difficulty by first applying \eqref{2D_spacetime_estimate_bound} in
  order to prove a conditional uniqueness result for \eqref{GP}. Then,
  we will use the rescaling \eqref{rescaling} and the correspondence
  of solutions to \eqref{GP1} and \eqref{GP} given by Lemma
  \ref{correspondence} in order to deduce a conditional uniqueness
  result for \eqref{GP1}. The details of this approach will be given
  in Subsection \ref{A conditional uniqueness result in 2D} below.
\end{remark}

We will now consider the endpoint case in $2D$. In other words, we set
$\alpha=\frac{1}{2}$. Before we prove Proposition
\ref{Lower_bound_2D}, let us first note some preliminaries. If
$\alpha=\frac{1}{2}$, then it is the case that:

\begin{equation}
  \notag
  I(\tau,p)=\sum_{m,n \in \mathbb{Z}^2} \frac{\widetilde{\delta}\big(\tau+Q(p)-2Q(n,m)\big) \cdot \langle p \rangle}{\langle m-p \rangle \cdot \langle n-p \rangle \cdot \langle p-n-m \rangle}
\end{equation}
where we recall that $\widetilde{\delta}=\chi_{[0,1]}$ is the
characteristic function of the interval $[0,1]$.

Let us now show that \emph{$I(\tau,p)$ is not uniformly bounded in
  $\tau$ and $p$ in the endpoint case.}  In order to do this, let
$\kappa \gg 1$ be an integer and let $p:=(\kappa,0)$. We choose $\tau
\in \mathbb{R}$ such that $|\tau+Q(p)| \leq 2$. Finally, we consider
only the part of the sum $I(\tau,p)$ in which $n=p$. We note that we
are then summing over all $m=(m_1,m_2) \in \mathbb{Z}^2$ such that:
\begin{equation}
  \notag
  \tau+Q(p)-2Q(p,m) \in [0,1]
\end{equation}
\begin{equation}
  \notag
  \Rightarrow -2Q(p,m) \in [-2,3]
\end{equation}
\begin{equation}
  \notag
  \Rightarrow \theta_1^2 \kappa \cdot m_1 \in [-2,1].
\end{equation}
If $\kappa$ is chosen to be sufficiently large, it follows that
$m_1=0$. In particular, for $p=(\kappa,0)$ and $|\tau+Q(p)| \leq 1$,
it follows that:
\begin{equation}
  \label{I_tau_p_endpoint_2D}
  I(\tau,p) \gtrsim \sum_{m_2 \in \mathbb{Z}} \frac{\kappa}{\sqrt{1+\kappa^2+m_2^2} \cdot \sqrt{1+m_2^2}} \gtrsim \int_{-\infty}^{+\infty} \frac{\kappa}{\sqrt{1+\kappa^2+x^2} \cdot \sqrt{1+x^2}}\,dx
\end{equation}
\begin{equation}
  \notag
  \gtrsim \int_{1 \leq |x| \leq \kappa} \frac{\kappa}{\sqrt{1+\kappa^2+x^2} \cdot \sqrt{1+x^2}} \,dx \gtrsim \int_{1 \leq |x| \leq\kappa} \frac{dx}{|x|} \sim \ln \kappa.
\end{equation}

\begin{remark}
  \label{sequence_cm2}
  We know from the calculation in \eqref{I_tau_p_endpoint_2D} and
  duality that there exists a sequence $(c_{m_2}) \in
  \ell^2(\mathbb{Z})$ such that $c_{m_2} \geq 0$ for all $m_2 \in
  \mathbb{Z}$ and:
  \begin{equation}
    \label{sequence_cm2_bound}
    \sum_{m_2 \in \mathbb{Z}} \frac{\kappa^{\frac{1}{2}}}{(1+\kappa^2+m_2^2)^{\frac{1}{4}} \cdot (1+m_2^2)^{\frac{1}{4}}} \cdot c_{m_2} \gtrsim \sqrt{\ln \kappa} \cdot (\sum_{m_2 \in \mathbb{Z}} c_{m_2}^2)^{\frac{1}{2}}
  \end{equation}
\end{remark}
In what follows, we will directly use the lower bound obtained in
\eqref{sequence_cm2_bound} to show that the spacetime estimate from
Proposition \ref{2D_spacetime_estimate} does not hold in the endpoint
case $\alpha=\frac{1}{2}$. In other words, we will not directly refer
to the fact that $I(\tau,p)$ is not uniformly bounded in $\tau$ and
$p$. Unlike in the arguments on the classical torus
(c.f. \cite[Proposition 3.12]{GSS}), where it was sufficient to get a
pointwise lower bound in the $\tau$ variable (and hence it was
possible to directly apply the unboundedness property of $I$), in the
setting of the general torus, we need to integrate in the $\tau$
variable over a finite interval and then estimate the obtained
integral from below. The reason for this change is the fact that
\emph{$\mathcal{U}_Q^{(k)}(t)$ is in general no longer periodic in
  time}. However, in the analysis, we can reduce to the case where the
integrand in the $\tau$ variable is bounded from below uniformly on a
finite interval. This is done in \eqref{tau1_integral} below. We can
then bound the obtained integral by using
\eqref{sequence_cm2_bound}. This is done in
\eqref{sequence_cm2_bound_application} below.

Let $\zeta \in L^1(\mathbb{R})$ be a function such that
$\widehat{\zeta} \geq 0$ on $\mathbb{R}$ and $\widehat{\zeta} \geq 1$
on $[-1,1]$. Such a function can be shown to exist (for example by
applying Lemma \ref{Fourier_Analysis_Lemma} below).  For
$\gamma^{(2)}_0: \Lambda^2 \times \Lambda^2 \rightarrow \mathbb{C}$,
we compute:

\begin{align*}
  & \big(\zeta(t)S^{(1,\frac{1}{2})}B_{1,2}\,\mathcal{U}_Q^{(2)}(t)\,\gamma^{(2)}_0\big)\,\widetilde{\,}\,(\tau,p;q)\\
  = &\big(\zeta(t)S^{(1,\frac{1}{2})}B^{+}_{1,2}\,\mathcal{U}_Q^{(2)}(t)\,\gamma^{(2)}_0\big)\,\widetilde{\,}\,(\tau,p;q)-\big(\zeta(t)S^{(1,\frac{1}{2})}B^{-}_{1,2}\,\mathcal{U}_Q^{(2)}(t)\,\gamma^{(2)}_0\big)\,\widetilde{\,}\,(\tau,p;q)\\
  =&\langle p \rangle^{\frac{1}{2}} \cdot \langle q
  \rangle^{\frac{1}{2}} \cdot \sum_{m,n \in \mathbb{Z}^2} \int
  [d^+_{n,m}(\tau_1,p;q)- d^-_{n,m}(\tau_1,p;q)]
  \widehat{\zeta}(\tau-\tau_1)d\tau_1
\end{align*}
where
\begin{align*}
  d^+_{n,m}(\tau_1,p;q):=&\widetilde{\delta}\big(\tau_1-Q(q)+Q(p)-2Q(n,m)\big) \cdot (\gamma^{(2)}_0)\,\widehat{\,}\,(p-m,p-n;q,p-n-m),\\
  d^-_{n,m}(\tau_1,p;q):=&\widetilde{\delta}\big(\tau_1+Q(p)-Q(q)-2Q(n,m)\big)
  \cdot (\gamma^{(2)}_0)\,\widehat{\,}\,(p,q-n-m;q-n,q-m).
\end{align*}
Here, and in the discussion that follows, $\widetilde{\,\cdot}$
applied to density matrices denotes the spacetime Fourier transform.
Let us now choose a specific $\gamma^{(2)}_0$. With the sequence
$(c_{m_2})$ as in Remark \ref{sequence_cm2}, we choose
$\gamma^{(2)}_0$ such that:

\begin{equation}
  \label{gamma^2_0}
  \langle (\kappa,-m_2) \rangle^{\frac{1}{2}} \cdot \langle (0,-m_2) \rangle^{\frac{1}{2}} \cdot (\gamma^{(2)}_0)\,\widehat{\,}\,\Big((\kappa,-m_2),(0,0);(0,0),(0,-m_2)\Big)=c_{m_2},
\end{equation}
for all $m_2 \in \mathbb{Z}$, and such that
$(\gamma^{(2)}_0)\,\widehat{\,}\,=0$ at all frequencies which are not
of the form $\Big((\kappa,-m_2),(0,0);(0,0),(0,-m_2)\Big)$ for some
$m_2 \in \mathbb{Z}$.

Let us now fix $\kappa \gg 1$ sufficiently large as above and let us
take $\bar{p}:=(\kappa,0), \bar{q}:=(0,0)$. Furthermore, we choose
$\bar{\tau} \in \mathbb{R}$ such that
$|\bar{\tau}-Q(\bar{q})+Q(\bar{p})|=|\bar{\tau}+Q(\bar{p})| \leq
1$. Since $(\gamma^{(2)}_0)\,\widehat{\,}\, \geq 0$, $\widehat{\zeta}
\geq 0$ on $\mathbb{R}$, and $\widehat{\zeta} \geq 1$ on $[-1,1]$, it
follows that:
\begin{align*}
  &\big(\zeta(t)S^{(1,\frac{1}{2})}B^{+}_{1,2}\,\mathcal{U}_Q^{(2)}(t)\,\gamma^{(2)}_0\big)\,\widetilde{\,}\,(\bar{\tau},\bar{p};\bar{q}) \\
  \gtrsim &\int_{|\tau_1-\bar{\tau}| \leq 1} \sum_{m,n \in
    \mathbb{Z}^2} d_{n,m}^+(\tau_1,\bar{p},\bar{q}) \cdot \langle
  \bar{p} \rangle^{\frac{1}{2}} \cdot \langle \bar{q}
  \rangle^{\frac{1}{2}}d\tau_1
\end{align*}
In the above expression, we are integrating over the set of $\tau_1$
for which $|\tau_1+Q(\bar{p})|\leq 2$. By our choice of $\kappa \gg 1$
as before, it follows that the above integrand is:
\begin{align}
  \label{tau1_integral}
  \gtrsim & \sum_{m_2 \in \mathbb{Z}} \int_{|\tau_1-\bar{\tau}| \leq 1} \kappa^{\frac{1}{2}} \cdot (\gamma^{(2)}_0)\,\widehat{\,}\,\big((\kappa,-m_2),(0,0);(0,0),(0,-m_2)\big)\,d \tau_1\\
  \gtrsim &\sum_{m_2 \in \mathbb{Z}} \frac{\kappa^{\frac{1}{2}} \cdot
    \langle (\kappa,-m_2) \rangle^{\frac{1}{2}} \cdot \langle (0,-m_2)
    \rangle^{\frac{1}{2}}}{(1+\kappa^2+m_2^2)^{\frac{1}{4}} \cdot
    (1+m_2^2)^{\frac{1}{4}}}\cdot
  (\gamma^{(2)}_0)\,\widehat{\,}\,\big((\kappa,-m_2),(0,0);(0,0),(0,-m_2)\big)\notag\\
  =&\sum_{m_2 \in \mathbb{Z}} \frac{\kappa^{\frac{1}{2}}}{(1+\kappa^2+m_2^2)^{\frac{1}{4}} \cdot (1+m_2^2)^{\frac{1}{4}}}  \cdot c_{m_2}\notag\\
  \label{sequence_cm2_bound_application}
  \gtrsim & \sqrt{\ln \kappa} \cdot (\sum_{m_2 \in \mathbb{Z}}
  c_{m_2}^2)^{\frac{1}{2}}
\end{align}
by using \eqref{sequence_cm2_bound} from Remark \ref{sequence_cm2}.
By construction of $\gamma^{(2)}_0$, it follows that the above
expression equals $\sqrt{\ln \kappa} \cdot
\|S^{(2,\frac{1}{2})}\gamma^{(2)}_0\|_{L^2(\Lambda^2 \times
  \Lambda^2)}$. Consequently, for all $\bar{\tau}$ with
$|\bar{\tau}+Q(\bar{p})| \leq 1$, it is the case that:
\begin{equation}
  \label{B+12_2D}
  \big(\zeta(t)S^{(1,\frac{1}{2})}B_{1,2}^{+}\,\mathcal{U}_Q^{(2)}(t)\,\gamma^{(2)}_0\big)\,\widetilde{\,}\,(\bar{\tau},\bar{p};\bar{q}) \gtrsim \sqrt{\ln \kappa} \cdot \|S^{(2,\frac{1}{2})} \gamma^{(2)}_0\|_{L^2(\Lambda^2 \times \Lambda^2)}.
\end{equation}

Let us now look at the contribution from $B_{1,2}^{-}$.
\begin{align*}
  &\big(\zeta(t) S^{(1,\frac{1}{2})} B_{1,2}^{-}\,\mathcal{U}_Q^{(2)}(t)\,\gamma^{(2)}_0\big)\,\widetilde{\,}\,(\bar{\tau},\bar{p};\bar{q})\\
  =& \langle \bar{p} \rangle^{\frac{1}{2}} \cdot \langle \bar{q}
  \rangle^{\frac{1}{2}} \cdot \sum_{m,n \in \mathbb{Z}^2} \int
  d^-_{n,m}(\tau_1,\bar{p};\bar{q}) \cdot
  \widehat{\zeta}(\bar{\tau}-\tau_1) \, d\tau_1
\end{align*}
By construction, the summand corresponding to $(m,n) \in \mathbb{Z}^2
\times \mathbb{Z}^2$ equals zero unless:
\begin{equation}
  \notag
  \begin{cases}
    \bar{p}=(\kappa,-\widetilde{m}_2)\\
    \bar{q}-m-n=(0,0)\\
    \bar{q}-n=(0,0)\\
    \bar{q}-m=(0,-\widetilde{m}_2)
  \end{cases}
\end{equation}
for some $\widetilde{m}_2 \in \mathbb{Z}$. In particular, it must be
the case that $m=n=\widetilde{m}_2=(0,0)$. Consequently:
\begin{align*}
  &  \big(\zeta(t) S^{(1,\frac{1}{2})} B_{1,2}^{-}\,\mathcal{U}_Q^{(2)}(t)\,\gamma^{(2)}_0\big)\,\widetilde{\,}\,(\bar{\tau},\bar{p};\bar{q})\\
  = &\langle \bar{p} \rangle^{\frac{1}{2}} \cdot \langle \bar{q}
  \rangle^{\frac{1}{2}} \cdot \int
  \widetilde{\delta}\big(\tau_1+Q(\bar{p})\big) \cdot
  (\gamma^{(2)}_0)\,\widehat{\,}\,\big(\bar{p},(0,0);(0,0),(0,0)\big)
  \cdot \widehat{\zeta}(\bar{\tau}-\tau_1) \, d\tau_1.
\end{align*}
Hence:
\begin{equation}
  \label{B-12_2D}
  \big|\big(\zeta(t)S^{(1,\frac{1}{2})}B_{1,2}^{-}\,\mathcal{U}_Q^{(2)}(t)\,\gamma^{(2)}_0\big)\,\widetilde{\,}\,(\bar{\tau},\bar{p};\bar{q})\big| \lesssim \|S^{(2,\frac{1}{2})} \gamma^{(2)}_0\|_{L^2(\Lambda^2 \times \Lambda^2)}.
\end{equation}

From \eqref{B+12_2D} and \eqref{B-12_2D}, it follows that:
\begin{equation}
  \label{B12_2D}
  \big|\big(\zeta(t)S^{(1,\frac{1}{2})}B_{1,2}\,\mathcal{U}_Q^{(2)}(t)\,\gamma^{(2)}_0\big)\,\widetilde{\,}\,(\bar{\tau},\bar{p};\bar{q})\big| \gtrsim \sqrt{\ln \kappa} \cdot \|S^{(2,\frac{1}{2})} \gamma^{(2)}_0\|_{L^2(\Lambda^2 \times \Lambda^2)}.
\end{equation}
We recall that $\kappa \gg 1$ is chosen to be sufficiently large
$\bar{p}:=(\kappa,0), \bar{q}:=(0,0)$ and $|\bar{\tau}+Q(\bar{p})|
\leq 1$. By Plancherel's Theorem and \eqref{B12_2D}, it follows that:

\begin{align}
  \notag &
  \|\zeta(t)S^{(1,\frac{1}{2})}B_{1,2}\,\mathcal{U}_Q^{(2)}(t)\,\gamma^{(2)}_0\|_{L^2(\mathbb{R}
    \times \Lambda \times \Lambda)}\\
  \sim
  &\|(\zeta(t)S^{(1,\frac{1}{2})}B_{1,2}\,\mathcal{U}_Q^{(2)}(t)\,\gamma^{(2)}_0)\,\widetilde{\,}\,\|_{L^2(\mathbb{R})
    \ell^2(\mathbb{Z}^2 \times \mathbb{Z}^2)} \notag \\
  \label{Endpoint_estimate1_2D}
  \gtrsim & \sqrt{\ln \kappa} \cdot
  \|S^{(2,\frac{1}{2})}\gamma^{(2)}_0\|_{L^2(\Lambda^2 \times
    \Lambda^2)}.
\end{align}
We can now prove Proposition \ref{Lower_bound_2D}.

\begin{proof}[Proof of Proposition \ref{Lower_bound_2D}] Let
  $\gamma^{(2)}_0$ be as in \eqref{gamma^2_0}. We take $\kappa \gg 1$
  as above and $\delta$ and $\zeta$ as in Lemma
  \ref{Fourier_Analysis_Lemma}.  The lemma follows from
  \eqref{Endpoint_estimate1_2D} as well as from the fact that
  $\supp\,\zeta \subseteq [-\delta,\delta]$ and $|\zeta|
  \lesssim_{\,\delta} 1$. We note that $\delta$ is indeed independent
  of $\kappa$.
\end{proof}

\subsection{A conditional uniqueness result}
\label{A conditional uniqueness result in 2D}

Let us fix $\alpha>\frac{1}{2}$. Let us first consider the following
class of density matrices on the torus $\Lambda=\mathbb{T}^2$:

\begin{definition}
  \label{mathcalA}
  Let $\mathcal{A}$ denote the class of all time-dependent sequences
  $\Gamma(t)=(\gamma^{(k)}(t))$, where each $\gamma^{(k)}:
  \mathbb{R}_t \times \Lambda^k \times \Lambda^k \rightarrow
  \mathbb{C}$ satisfies:

  \begin{itemize}
  \item[i)]
    $\gamma^{(k)}(t,x_{\sigma(1)},\ldots,x_{\sigma(k)};x_{\sigma(1)}',\ldots,x_{\sigma(k)}')=\gamma^{(k)}(t,x_1,\ldots,x_k;x_1',\ldots,x_k')$,
    \newline for all $t \in \mathbb{R}, x_1,\ldots,x_k, x'_1, \ldots,
    x'_k \in \Lambda$, and for all $\sigma \in S^k$.
  \item[ii)] There exist positive and continuous functions $f,g :
    \mathbb{R} \rightarrow \mathbb{R}$, which are independent of $k$,
    such that for all $t \in \mathbb{R}$ and for all $j \in
    \{1,2,\ldots,k\}$:
    \begin{equation}
      \notag
      \int_{t-g(t)}^{t+g(t)} \|S^{(k,\alpha)}B_{j,k+1}(\gamma^{(k+1)})(s)\|_{L^2(\Lambda^k \times \Lambda^k)} ds \leq f^{k+1}(t).
    \end{equation}
  \end{itemize}
\end{definition}
For future reference, we will define the class $\mathcal{A}$
analogously on $\mathbb{T}^d$ for $d \geq 2$.

It is possible to argue as in \cite{GSS,KSS,KM} to deduce that:
\begin{proposition}
  \label{Lambda_conditional_uniqueness}
  Solutions to the Gross-Pitaevskii hierarchy on $\Lambda$ with a
  modified Laplacian \eqref{GP} are unique in the class
  $\mathcal{A}$. More precisely, given two solutions in $\mathcal{A}$
  with the same initial data, these two solutions have to be equal.
\end{proposition}

\begin{proof}
  Let us observe that the boardgame argument from \cite{KM} still
  applies on $\Lambda=\mathbb{T}^2$ with the Laplacian
  $\Delta_Q$. Namely, in the boardgame argument, one interchanges the
  different $\Lambda$ variables without interchanging their
  components. Hence, the fact that the operator $\Delta_Q$ acts
  differently in each component does not affect the argument. The rest
  of the proof then follows by using the spacetime bound from
  Proposition \ref{2D_spacetime_estimate} analogously as in
  \cite{GSS,KSS,KM}. We will omit the details.
\end{proof}

We are more interested in obtaining uniqueness for the
Gross-Pitaevskii hierarchy \eqref{GP1} on $\Lambda_2$.  Let us recall
the scaling transformation given by \eqref{rescaling}. In this way, we
obtain two sequences $\Gamma(t)=(\gamma^{(k)}(t))$, a sequence of
density matrices on $\Lambda$ and
$\widetilde{\Gamma}(t)=(\widetilde{\gamma}^{(k)}(t))$, a sequence of
density matrices on $\Lambda_2$. From Lemma \ref{correspondence}, we
know that $\Gamma(t)$ solves \eqref{GP} if and only if
$\widetilde{\Gamma}(t)$ solves \eqref{GP1}. Let us now note another
correspondence result between $\Gamma$ and $\widetilde{\Gamma}$:

\begin{lemma}
  \label{correspondence2}
  $\Gamma$ belongs to the class $\mathcal{A}$ if and only if
  $\widetilde{\Gamma}$ belongs to the class $\widetilde{\mathcal{A}}$.
\end{lemma}

\begin{proof}
  We will show that $\Gamma \in \mathcal{A}$ implies that
  $\widetilde{\Gamma} \in \widetilde{\mathcal{A}}$. The reverse
  implication is proved in an analogous way.  Suppose that $\Gamma \in
  \mathcal{A}$. We will show that $\widetilde{\Gamma} \in
  \widetilde{\mathcal{A}}$. The fact that condition $i)$ is satisfied
  is immediate.  We need to check the a priori bound given by
  $ii)$. In order to do this, we compute, for all $t$:

\begin{align*}
  &\big(B_{1,k+1}^{+} \widetilde{\gamma}^{(k+1)}\big)\,\widehat{\,}\,(t,\xi_1,\ldots,\xi_k;\xi'_1,\ldots,\xi'_k)\\
  =&\sum_{\eta,\eta' \in \theta_1 \mathbb{Z} \times \theta_2
    \mathbb{Z}}
  (\widetilde{\gamma}^{(k+1)})\,\widehat{\,}\,(t,\xi_1-\eta+\eta',
  \xi_2,\ldots,\xi_k, \eta;\xi'_1,\xi'_2,\ldots,\xi'_k,\eta')\\
  =&\sum_{\widetilde{\eta},\widetilde{\eta}\,' \in \mathbb{Z}^2}
  \frac{1}{(\theta_1 \theta_2)^{2k+2}}  \cdot (\gamma^{(k+1)})\,\widehat{\,}\,\Big(t,\widetilde{\xi}-\widetilde{\eta}+\widetilde{\eta}\,',\widetilde{\xi}_2,\ldots,\widetilde{\xi}_k,\widetilde{\eta};\widetilde{\xi}_1',\widetilde{\xi}_2',\ldots,\widetilde{\xi}_k',\widetilde{\eta}\,'\Big)\\
  =&\frac{1}{(\theta_1 \theta_2)^{2k+2}} \big(B_{1,k+1}^{+}
  \gamma^{(k+1)}\big)\,\widehat{\,}\,(t,\widetilde{\xi}_1,\ldots,\widetilde{\xi}_k;\widetilde{\xi}_1',\ldots,\widetilde{\xi}_k').
\end{align*}
Here, we have used \eqref{Fourier_transform_gamma_tilde}.
By an analogous calculation, it follows that:
\begin{align*}
  &\big(B_{j,k+1} \widetilde{\gamma}^{(k+1)}\big)\,\widehat{\,}\,(t,\xi_1,\ldots,\xi_k;\xi'_1,\ldots,\xi'_k)\\
  =& \frac{1}{(\theta_1 \theta_2)^{2k+2}} \big(B_{j,k+1}
  \gamma^{(k+1)}\big)\,\widehat{\,}\,(t,\widetilde{\xi}_1,\ldots,\widetilde{\xi}_k;\widetilde{\xi}_1',\ldots,\widetilde{\xi}_k'),
\end{align*}
for all $j \in \{1,2,\ldots,k\}$.  Consequently,
\begin{equation}
  \notag
  \|S^{(k,\alpha)} B_{j,k+1} \widetilde{\gamma}^{(k+1)}(t)\|_{L^2(\Lambda_2^k \times \Lambda_2^k)}  \leq C_1^k \cdot \|S^{(k,\alpha)} B_{j,k+1} \gamma^{(k+1)}(t)\|_{L^2(\Lambda^k \times \Lambda^k)} 
\end{equation}
for some constant $C_1$ which depends only on
$\alpha,\theta_1,\theta_2$. The claim now follows.
\end{proof}

\begin{remark}
  The result of Lemma \ref{correspondence2} extends in general to $d$
  dimensions.
\end{remark}

We can now deduce the main conditional uniqueness result of this
section:
\begin{theorem}
  \label{uniqueness_Lambda2_2D}
  Solutions to the Gross-Pitaevskii hierarchy \eqref{GP1} on
  $\Lambda_2$ are unique in the class $\widetilde{\mathcal{A}}$.
\end{theorem}

\begin{proof}
  In order to prove this fact, we first recall that $\Gamma(t)$ and
  $\widetilde{\Gamma}(t)$ are related by the scaling transformation
  \eqref{rescaling}. We then apply Lemma \ref{correspondence}, Lemma
  \ref{correspondence2}, and Proposition
  \ref{Lambda_conditional_uniqueness}. The claim follows.
\end{proof}

\subsection{A rigorous derivation of the defocusing cubic nonlinear
  Schr\"{o}dinger equation on a general two-dimensional torus}
\label{Rigorous derivation 2D irrational torus}
In this subsection, we will obtain a rigorous derivation of the
defocusing cubic nonlinear Schr\"{o}dinger equation from many-body
quantum systems on general two-dimensional tori. Let us recall that,
in \cite{KSS}, this result was obtained on the classical torus
$\Lambda=\mathbb{T}^2$. We will now extend it to the case of the
spatial domain $\Lambda_2$, which can, in particular, be an
\emph{irrational torus}.
We will prove the following result:
\begin{theorem}
  \label{NLS_Lambda2}
  Let $V \in W^{2,\infty}(\Lambda_2)$ be such that $V \geq 0$,
  $\int_{\Lambda_2} V(x)\,dx=b_0>0$, and let $\beta \in
  (0,\frac{3}{4})$. Suppose that $(\psi_N)_N \in
  \mathop{\bigoplus}_{N} L^2(\Lambda_2^N)$ satisfies the properties of
  bounded energy per particle \eqref{Bounded energy per particle} and
  asymptotic factorization \eqref{Asymptotic factorization}. Then,
  there exists a sequence $N_j \rightarrow \infty$ such that for all
  $t \in \mathbb{R}$ and $k \in \mathbb{N}$:
  \begin{equation}
    \notag
    Tr\big|\gamma^{(k)}_{N_j,t}-|\phi_t \rangle \langle \phi_t\big|^{\otimes k}\big| \rightarrow 0\,\,\,\mbox{as}\,\,j \rightarrow \infty,
  \end{equation}
  where $\phi_t$ solves the defocusing cubic nonlinear Schr\"{o}dinger
  equation on $\Lambda_2$ with initial data $\phi$:
  \begin{equation}
    \notag
    \begin{cases}
      i \partial_t \phi_t + \Delta \phi_t=b_0 |\phi_t|^2 \phi_t\\
      \phi_t \big|_{t=0}=\phi.
    \end{cases}
  \end{equation}
\end{theorem}

\begin{proof}
  The proof of Theorem \ref{NLS_Lambda2} follows from the arguments
  given in \cite{KSS} combined with the uniqueness result of Theorem
  \ref{uniqueness_Lambda2_2D}. Namely, we recall that the limiting
  arguments presented in \cite[Sections 3-6]{KSS} are originally given
  in the setting of the classical torus. Nevertheless, since these
  arguments do not depend on any Diophantine properties of the
  frequencies, but just on Sobolev embedding type results, they
  directly carry over to the setting of a general torus. More
  precisely, the only place in the analysis of \cite{KSS} where the
  authors use the rationality of the torus is in the proof of the
  conditional uniqueness result \cite[Theorem 7.4]{KSS}, whose
  analogue on a general torus we have proven in Theorem
  \ref{uniqueness_Lambda2_2D} above.

  In particular, the analogue of \cite[Theorem 5.2]{KSS} on
  $\Lambda_2$ holds. This result implies that the density matrices
  obtained according to the limiting procedure belong to the class
  $\widetilde{\mathcal{A}}$ for $\alpha<1$ \emph{with $\widetilde{f}$
    and $\widetilde{g}$ being positive constant
    functions}. Consequently, we obtain that the limit has to be the
  factorized solution $(|\phi_t \rangle \langle \phi_t |^{\otimes
    k})_k$. We refer the reader to \cite[Section 2]{KSS} for a more
  precise outline of this procedure.
\end{proof}

\section{The three-dimensional problem}
\label{3D_problem}
In this section, we will consider the three-dimensional
problem. Throughout the section, $\Lambda$ will denote the
three-dimensional classical torus $\mathbb{T}^3$. In Subsection
\ref{The spacetime estimate in three dimensions}, we will prove a
three-dimensional analogue of the sharp spacetime estimate proved in
Subsection \ref{The spacetime estimate in two dimensions} above. In
Subsection \ref{A conditional uniqueness result 3D}, we will prove the
corresponding conditional uniqueness result. In particular, this
extends the uniqueness result in \cite{GSS} to general tori.  Finally,
in Subsection \ref{An unconditional uniqueness result 3D}, we will
prove an unconditional uniqueness result, which will allow us to
obtain a rigorous derivation of the defocusing cubic NLS on the
irrational torus as was done in the setting of the classical torus in
\cite{VS2}.

\subsection{The spacetime estimate in three dimensions}
\label{The spacetime estimate in three dimensions}
We will now prove a conditional uniqueness result for the
three-dimensional problem. In particular, we will extend the
uniqueness result of \cite{GSS} to general three-dimensional
tori. We will start by proving a spacetime estimate, which is the
three-dimensional analogue of Proposition
\ref{2D_spacetime_estimate}.
\begin{proposition}
  \label{3D_spacetime_estimate}
  Let $\alpha>1$ be given. There exists $C>0$, which depends only on
  $\alpha, \theta_1, \theta_2$ such that, for all $k \in \mathbb{N}$,
  and for all $\gamma_0^{(k+1)}: \Lambda^{k+1} \times \Lambda^{k+1}
  \rightarrow \mathbb{C}$, the following estimate holds:
  \begin{equation}
    \notag
    \|S^{(k,\alpha)} B_{j,k+1} \,\mathcal{U}_Q^{(k+1)}(t)\,\gamma_0^{(k+1)}\|_{L^2([0,1] \times \Lambda^k \times \Lambda^k)} \leq C \|S^{(k+1,\alpha)}\gamma_0^{(k+1)}\|_{L^2(\Lambda^{k+1} \times \Lambda^{k+1})}.
  \end{equation}
\end{proposition}
As was the case in Proposition \ref{2D_spacetime_estimate}, the range
of regularity exponents $\alpha>1$ in Proposition
\ref{3D_spacetime_estimate} is sharp due to the following result:
\begin{proposition}
  \label{Lower_bound_3D}
  For $\kappa \in \mathbb{N}$ sufficiently large, there exists
  $\gamma^{(2)}_0: \Lambda^2 \times \Lambda^2 \rightarrow \mathbb{C}$,
  such that for $\delta>0$ sufficiently small:
  \begin{equation}
    \notag
    \|S^{(1,1)}B_{1,2}\,\mathcal{U}_Q^{(2)}(t)\,\gamma^{(2)}_0\|_{L^2([0,\delta] \times \Lambda \times \Lambda)} \gtrsim_{\,\delta} \sqrt{\ln \kappa} \cdot \|S^{(2,1)} \gamma^{(2)}_0\|_{L^2(\Lambda^2 \times \Lambda^2)}.
  \end{equation}
\end{proposition}
We will first prove Proposition \ref{3D_spacetime_estimate}.
\begin{proof}[Proof of Proposition \ref{3D_spacetime_estimate}]
  The proof will be similar to that of Proposition
  \ref{2D_spacetime_estimate}.  We will just outline the key
  differences. As in \eqref{I_tau_p2}, it suffices to obtain a uniform
  bound on:
  \begin{equation}
    \label{I_tau_p_3D}
    I(\tau,p)=\sum_{m,n \in \mathbb{Z}^3} \frac{\widetilde{\delta}\big(\tau+Q(p)-2Q(n,m)\big) \cdot \langle p \rangle^{2\alpha}}{\langle m-p \rangle^{2\alpha} \cdot \langle n-p \rangle^{2\alpha} \cdot \langle p-n-m \rangle^{2\alpha}},
  \end{equation}
  whenever $\alpha>1$. Here, the sum is over elements of
  $\mathbb{Z}^3$ and the notation has been adapted to the
  three-dimensional setting. Again, since $\alpha>1>\frac34$, the contributions of $m=0$ or $n=0$ to the sum are uniformly bounded.

  As before, given $j=(j_1,j_2,j_3) \in \mathbb{N}_0^3$, $\tau \in
  \mathbb{R}$, and $p \in \mathbb{Z}^2$, we let $E_{\tau,p}(j)
  \subseteq \mathbb{Z}^3 \setminus \{0\} \times \mathbb{Z}^3 \setminus
  \{0\}$
be the set of all pairs $(m,n)$ such that:
\begin{equation}
  \notag
  \begin{cases}
    \tau + Q(p)-2Q(n,m) \in [0,1]\\
    |m-p| \sim 2^{j_1}, |n-p| \sim 2^{j_2}, |p-n-m| \sim 2^{j_3}.
  \end{cases}
\end{equation}

The bound that we would like to prove in three dimensions is:
\begin{equation}
  \label{E_tau_p_bound_3D}
  \#E_{\tau,p}(j) \lesssim_{\epsilon} 2^{(2+\epsilon)j_{\min}+(2+\epsilon)j_{\med}}
\end{equation}
for all $\epsilon>0$. As in the proof of Proposition
\ref{2D_spacetime_estimate}, the claim follows.

Let us now prove \eqref{E_tau_p_bound_3D}. We note that the
three-dimensional analogue of \eqref{E_tau_p} holds:
\begin{equation}
  \label{E_tau_p_3D}
  \#E_{\tau,p}(j) \leq c \cdot \Big\|\mathop{\mathop{\sum_{\eta,\eta'
        \in \mathbb{Z}^3}}_{\eta \in C_{2^{j_1\vee
          j_2+1}}(0)\,\cap\,C_{2^{j_2\vee
          j_3+2}}(p)\,\cap\,C_{2^{j_1\vee j_3+2}}(-p)}}_{\eta' \in
    C_{2^{j_3}}(p)\,\cap\,C_{2^{j_1\vee  j_2+1}}(2p)} e^{\frac{1}{2}it\big(Q(\eta)-Q(\eta')\big)} \Big\|_{L^1_t(I)}
\end{equation}
for some $c>0$ and for some finite interval $I \subseteq \mathbb{R}$.
Similarly as in \eqref{E_tau_p}, in \eqref{E_tau_p_3D} $C_{2^j}(q)$
denotes a three-dimensional cube, centered at $q$, whose sides are
parallel to the coordinate axes, and who have sidelength $2^{j+1}$. As
before, we need to consider several cases, depending on the relative
sizes of $j_1, j_2, j_3$. We recall that $j_1,j_2,j_3$ are ordered as
$j_{\max} \geq j_{\med} \geq j_{\min}$.

\vspace{5mm} \textbf{Case 1:}\,\,$j_3=\min\{j_1,j_2,j_3\}$.
\vspace{5mm}

This case follows once we estimate the following analogue of the
expression in \eqref{Case1_Term1}:

\begin{equation}
  \notag
  \int_{I} \Big| \mathop{\sum_{\eta_1,\eta_2,\eta_3 \in \mathbb{Z}}}_{\eta_j \in I^{\med}_j} \mathop{\sum_{\eta'_1,\eta'_2,\eta'_3 \in \mathbb{Z}}}_{\eta'_j \in I^{\min}_j} e^{\,\frac{1}{2}it\big(\theta_1^2\eta_1^2\,+\,\theta_2^2\eta_2^2\,+\,\theta_3^2\eta_3^2\,-\,\theta_1^2(\eta'_1)^2\,-\,\theta_2^2(\eta'_2)^2\,-\,\theta_3^2(\eta'_3)^2\big)}\Big|dt.
\end{equation}
Here, $I^{\med}_1,I^{\med}_2, I^{\med}_3$ are fixed intervals of size
$\sim 2^{j_{\med}}$, and $I^{\min}_1,I^{\min}_2,I^{\min}_3$ are fixed
intervals of size $\sim 2^{j_{\min}}$.  By H\"{o}lder's inequality,
this expression is:
\begin{equation}
  \notag
  \leq \Big(\int_{I} \Big|\mathop{\sum_{\eta_1 \in \mathbb{Z}}}_{\eta_1 \in I^{\med}_1} e^{\,\frac{1}{2}it\theta_1^2 \eta_1^2}\Big|^6\,dt\Big)^{\frac{1}{6}} \cdot \Big(\int_{I}\Big|\mathop{\sum_{\eta_2 \in \mathbb{Z}}}_{\eta_2 \in I^{\med}_2} e^{\,\frac{1}{2}it\theta_2^2\eta_2^2}\Big|^6\,dt\Big)^{\frac{1}{6}} \cdot \Big(\int_{I}\Big|\mathop{\sum_{\eta_3 \in \mathbb{Z}}}_{\eta_3 \in I^{\med}_3} e^{\,\frac{1}{2}it\theta_3^2\eta_3^2}\Big|^6\,dt\Big)^{\frac{1}{6}} \cdot 
\end{equation}
\begin{equation}
  \notag
  \cdot \Big(\int_{I} \Big|\mathop{\sum_{\eta'_1 \in \mathbb{Z}}}_{\eta'_1 \in I^{\min}_1} e^{\,\frac{1}{2} it\theta_1^2(\eta'_1)^2}\Big|^6\,dt\Big)^{\frac{1}{6}} \cdot \Big(\int_{I} \Big|\mathop{\sum_{\eta'_2 \in \mathbb{Z}}}_{\eta'_2 \in I^{\min}_2} e^{\,\frac{1}{2} it\theta_2^2 (\eta'_2)^2}\Big|^6\,dt\Big)^{\frac{1}{6}} \cdot \Big(\int_{I} \Big|\mathop{\sum_{\eta'_3 \in \mathbb{Z}}}_{\eta'_3 \in I^{\min}_3} e^{\,\frac{1}{2} it\theta_3^2 (\eta'_3)^2}\Big|^6\,dt\Big)^{\frac{1}{6}}
  .
\end{equation}
For each of the six factors, we rescale in time and apply Corollary
\ref{cor:lp} with $p=6$ to deduce that this product is:
\begin{equation}
  \notag
  \lesssim_{\epsilon} \big(2^{\,(4+\epsilon)j_{\med}}\big)^{\frac{1}{6}} \cdot \big(2^{\,(4+\epsilon)j_{\med}}\big)^{\frac{1}{6}} \cdot \big(2^{\,(4+\epsilon)j_{\med}}\big)^{\frac{1}{6}} \cdot \big(2^{\,(4+\epsilon)j_{\min}}\big)^{\frac{1}{6}} \cdot \big(2^{\,(4+\epsilon)j_{\min}}\big)^{\frac{1}{6}} \cdot \big(2^{\,(4+\epsilon)j_{\min}}\big)^{\frac{1}{6}}.
\end{equation}
\begin{equation}
  \notag
  \lesssim 2^{\,(2+\epsilon)j_{\min}\,+\,(2+\epsilon)j_{\med}}
\end{equation}
for all $\epsilon>0$. This is a good bound.

\vspace{5mm} \textbf{Case 2:}\,\,$j_1=\min\{j_1,j_2,j_3\}$.
\vspace{5mm}

Given $k=(k_1,k_2,k_3) \in \mathbb{Z}^3$, we let:
\begin{equation}
  \notag
  B_k:=[2^{j_1}k_1,2^{j_1}k_1+2^{j_1}-1] \times [2^{j_1}k_2,2^{j_1}k_2+2^{j_1}-1] \times [2^{j_1}k_3,2^{j_1}k_3+2^{j_1}-1].
\end{equation}
By using the arguments from Case 1, it follows that for all $(k,k')
\in \mathbb{Z}^3 \times \mathbb{Z}^3$:
\begin{equation}
  \notag
  \#\big(E_{\tau,p}(j) \cap (B_k \times B_{k'})\big) \lesssim_{\epsilon} 2^{\,(4+\epsilon)j_1}
\end{equation}
for all $\epsilon>0$. Hence, in order to prove
\eqref{E_tau_p_bound_3D} in this case, it suffices to show that:
\begin{equation}
  \notag
  \{(k,k') \in \mathbb{Z}^3 \times \mathbb{Z}^3, \,E_{\tau,p}(j) \cap (B_k \times B'_k) \neq \emptyset\} \lesssim 2^{\,2\,\min\{j_2,j_3\}-2j_1}.
\end{equation}
By applying \cite[Lemma 3.6]{GSS} when $D=6$, we observe that:
\begin{equation}
  \notag
  \#\{(k,k') \in \mathbb{Z}^2 \times \mathbb{Z}^2,\,E_{\tau,p}(j) \cap (B_k \times B_{k'}) \neq \emptyset\} \lesssim 2^{-6j_1}\big|X_{2^{j_1}}\big|.
\end{equation}
Here, $X_{2^{j_1}}$ denotes the set of all points in $\mathbb{R}^3
\times \mathbb{R}^3$ which are of distance at most
$2^{j_1+\frac{1}{2}}$ from the set $E_{\tau,p}(j)$. Namely, when
$D=6$, the fact that $E_{\tau,p}(j) \cap (B_k \times B_{k'}) \neq
\emptyset$ implies that $(k,k')$ belongs to a $2^{j_1+\frac{1}{2}}$
thickening of $E_{\tau,p}(j)$.

Thus, we would like to show that:
\begin{equation}
  \label{X2j_1_3D}
  \big|X_{2^{j_1}}\big| \lesssim 2^{\,4j_1\,+\,2\,\min\{j_2,j_3\}}.
\end{equation}

Let $(m,n) \in E_{\tau,p}(j)$.  As in the two-dimensional problem, it
is the case that $m$ is allowed to vary over a ball of radius
$O(2^{j_1})$, and that for fixed $m$, $n$ varies over a ball of radius
$O(2^{\,\min\{j_2,j_3\}})$. The fact that $\tau+Q(p)-2Q(n,m) \in
[0,1]$ can be rewritten as:
\begin{equation}
  \notag
  \frac{\tau+Q(p)}{2 |m|}-n_1 \cdot \frac{\theta_1^2 \, m_1}{|m|}-n_2 \cdot \frac{\theta_2^2 \, m_2}{|m|}-n_3 \cdot \frac{\theta_3^2\,m_3}{|m|} \in \Big[0,\frac{1}{2|m|}\Big]
\end{equation}
Hence, $n$ lies within an $O(1)$ distance of a fixed plane in
$\mathbb{R}^3$.

Consequently, given $(x,y)$, which belongs to the thickening
$X_{2^{j_1}}$, it follows that $x$ lies in a ball of radius
$O(2^{j_1})$, and for a fixed $x$, the $y$ coordinate lies in the
intersection of a ball of radius $O(2^{\,\min\{j_2,j_3\}})$ and an
$O(2^{j_1})$ neighborhood of a fixed plane. It follows that the
Lebesgue measure of the set to which $x$ is localized is $\lesssim
2^{3j_1}$ and that, for a fixed $x$, the Lebesgue measure of the set
to which $y$ is localized is $\lesssim
2^{j_1+\,2\min\{j_2,j_3\}}$. The bound \eqref{X2j_1_3D} now follows.
 
\vspace{5mm} \textbf{Case 3:}\,\,$j_2=\min\{j_1,j_2,j_3\}$.
\vspace{5mm}

Case 3 is analogous to Case 2 due to the symmetry in $m$ and $n$ in
the definition of $I(\tau,p)$ given in \eqref{I_tau_p_3D}.
\end{proof}

Let us now prove Proposition \ref{Lower_bound_3D}. The proof will be
very similar to the proof of Proposition \ref{Lower_bound_2D}, so we
will just outline the main differences.
\begin{proof}[Proof of Proposition \ref{Lower_bound_3D}]
  If $\alpha=1$, then:
\begin{equation}
  \notag
  I(\tau,p)=\sum_{m,n \in \mathbb{Z}^3} \frac{\widetilde{\delta}\big(\tau+Q(p)-2Q(n,m)\big) \cdot \langle p \rangle^2}{\langle m-p \rangle^2 \cdot \langle n-p \rangle^2 \cdot \langle p-n-m \rangle^2}.
\end{equation}
We will show that $I(\tau,p)$ is not uniformly bounded in
$\tau,p$. Similarly as in the $2D$ setting, we let $\kappa \gg 1$ be
an integer, we let $p:=(\kappa,0,0)$, and we consider only the part of
the sum $I(\tau,p)$ in which $n=p$. Hence, we sum over all
$m=(m_1,m_2,m_3) \in \mathbb{Z}^3$ such that $\tau+Q(p)-2Q(n,m) \in
[0,1]$. As in the $2D$ setting, it follows that for $\kappa$
sufficiently large, in this sum, $m_1=0$.  Hence:

\begin{align*}
  I(\tau,p) 
\gtrsim & \sum_{m_2,m_3 \in \mathbb{Z}}
\frac{\kappa^2}{(1+\kappa^2+m_2^2+m_3^2) \cdot (1+m_2^2+m_3^2)} \\
\sim & \int_{\mathbb{R}^2} \frac{\kappa^2}{(1+\kappa^2+|x|^2) \cdot (1+|x|^2)}\,dx \gtrsim \ln \kappa.
\end{align*}
The details of the above calculation can be found in the proof of
\cite[Lemma 3.11]{GSS}.  In particular, by duality, there exists a
sequence $(d_{\,m_2,m_3}) \in \ell^2(\mathbb{Z}^2)$ such that
$d_{\,m_2,m_3} \geq 0$ for all $(m_2,m_3) \in \mathbb{Z}^2$ and:
\begin{align*}
  & \sum_{m_2,m_3 \in \mathbb{Z}}
  \frac{\kappa}{\sqrt{1+\kappa^2+m_2^2+m_3^2} \cdot
    \sqrt{1+m_2^2+m_3^2}} \cdot d_{\,m_2,m_3} \\
\gtrsim & \sum_{m_2,m_3 \in \mathbb{Z}} \sqrt{\ln \kappa} \cdot (\sum_{m_2,m_3} d_{\,m_2,m_3}^2)^{\frac{1}{2}}.
\end{align*}
We now choose a specific $\gamma^{(2)}_0: \Lambda^2 \times \Lambda^2
\rightarrow \mathbb{C}$, similarly as in \eqref{gamma^2_0}. In
particular, with $(d_{\,m_2,m_3})$ as above, we choose
$\gamma^{(2)}_0$ such that:
\begin{align*}
  &\langle (\kappa,-m_2,-m_3) \rangle \cdot \langle (0,-m_2,-m_3)
  \rangle \cdot
  (\gamma^{(2)}_0)\,\widehat{\,}\,\Big((\kappa,-m_2,-m_3),(0,0);(0,0),(0,-m_2,-m_3)\Big)\\
&=d_{\,m_2,m_3},
\end{align*}
for all $m_2,m_3 \in \mathbb{Z}$, and such that
$(\gamma^{(2)}_0)\,\widehat{\,}\,=0$ at all frequencies which are not
of the form $\Big((\kappa,-m_2,-m_3),(0,0);(0,0),(0,-m_2,-m_3)\Big)$
for some $m_3 \in \mathbb{Z}$.  Let $\zeta \in L^1(\mathbb{R})$ be
such that $\widehat{\zeta} \geq 0$ on $\mathbb{R}$ and
$\widehat{\zeta} \geq 1$ on
$[-1,1]$, cp.\ Lemma \ref{Fourier_Analysis_Lemma}.
Then, by arguing as in the proof of \eqref{Endpoint_estimate1_2D}
\begin{equation}
  \notag
  \|\zeta(t)S^{(1,1)}B_{1,2}\,\mathcal{U}_Q^{(2)}(t)\,\gamma^{(2)}_0\|_{L^2(\mathbb{R} \times \Lambda \times \Lambda)} 
  \gtrsim \sqrt{\ln \kappa} \cdot \|S^{(2,1)}\gamma^{(2)}_0\|_{L^2(\Lambda^2 \times \Lambda^2)}.
\end{equation}
The proposition now follows.
\end{proof}

\begin{remark}
  \label{higher_dimensions}
  More generally, in $d \geq 2$ dimensions, it is the case that for
  all $k \in \mathbb{N}$ and $\gamma_0^{(k+1)}:\Lambda^{k+1} \times
  \Lambda^{k+1} \rightarrow \mathbb{C}$:
  \begin{equation}
    \label{d_dimensional_spacetime_estimate_bound}
    \|S^{(k,\alpha)} B_{j,k+1} \,\mathcal{U}_Q^{(k+1)}(t)\,\gamma_0^{(k+1)}\|_{L^2([0,1] \times \Lambda^k \times \Lambda^k)} \leq C \|S^{(k+1,\alpha)}\gamma_0^{(k+1)}\|_{L^2(\Lambda^{k+1} \times \Lambda^{k+1})}.
  \end{equation}
  whenever $\alpha>\frac{d-1}{2}$. Here, $\Lambda=\mathbb{T}^d$ is the
  classical $d$-dimensional torus and $\mathcal{U}_Q^{(k)}(t)$ is the
  free evolution operator obtained from the $d$-dimensional version of
  $\Delta_Q$. The constant $C>0$ depends on $d$ and $\alpha$.

  Moreover, for $\kappa \in \mathbb{N}$ sufficiently large, there
  exists $\gamma^{(2)}_0: \Lambda^2 \times \Lambda^2 \rightarrow
  \mathbb{C}$, such that for $\delta>0$ sufficiently small:
  \begin{equation}
    \label{d_dimensional_lower_bound}
    \|S^{(1,\frac{d-1}{2})}B_{1,2}\,\mathcal{U}_Q^{(2)}(t)\,\gamma^{(2)}_0\|_{L^2([0,\delta] \times \Lambda \times \Lambda)} \gtrsim_{\,\delta} \sqrt{\ln \kappa} \cdot \|S^{(2,\frac{d-1}{2})} \gamma^{(2)}_0\|_{L^2(\Lambda^2 \times \Lambda^2)}.
  \end{equation}

  The estimate \eqref{d_dimensional_spacetime_estimate_bound} is
  proved by using the geometric arguments given in Proposition
  \ref{2D_spacetime_estimate} and Proposition
  \ref{3D_spacetime_estimate}, and applying Corollary
  \ref{cor:lp}. The estimate \eqref{d_dimensional_lower_bound} is
  proved by using the same methods as in Proposition
  \ref{Lower_bound_2D} and Proposition \ref{Lower_bound_3D}, and the
  fact that:
  \begin{equation}
    \label{ln_kappa_d_dimensions}
    \sum_{m_2,\ldots,m_d \in \mathbb{Z}} \frac{\kappa^{d-1}}{(1+\kappa^2+m_2^2+\cdots+m_d^2)^{\frac{d-1}{2}} \cdot (1+m_2^2+\cdots+m_d^2)^{\frac{d-1}{2}}} \gtrsim \ln \kappa.
  \end{equation}
  The estimate \eqref{ln_kappa_d_dimensions} follows by using polar
  coordinates to see that the left-hand side is:
  \begin{equation}
    \notag
    \sim \int_{\mathbb{R}} \frac{\kappa^{d-1} \cdot r^{d-2}}{(1+\kappa^2+r^2)^{\frac{d-1}{2}} \cdot (1+r^2)^{\frac{d-1}{2}}}\,dr \gtrsim \int_{1+r^2 \leq \kappa^2} \frac{\kappa^{d-1} \cdot r^{d-2}}{\kappa^{d-1} \cdot r^{d-1}}\,dr \gtrsim \ln \kappa.
  \end{equation}
  We will omit the details of the proofs of
  \eqref{d_dimensional_spacetime_estimate_bound} and
  \eqref{d_dimensional_lower_bound}.
\end{remark}

\subsection{A conditional uniqueness result}
\label{A conditional uniqueness result 3D}
We recall the classes $\widetilde{\mathcal{A}}$ and $\mathcal{A}$
given in Definitions \ref{mathcalAtilde} and \ref{mathcalA}
respectively. Here, we are considering $\widetilde{\mathcal{A}}$ and
$\mathcal{A}$ in three dimensions. By arguing as in the proof of
Theorem \ref{uniqueness_Lambda2_2D}, we can deduce the following
three-dimensional result:

\begin{theorem}
  \label{uniqueness_Lambda3_3D}
  Solutions to the Gross-Pitaevskii hierarchy on $\Lambda_3$ are
  unique in the class $\widetilde{\mathcal{A}}$ whenever
  $\alpha>1$. Moreover, whenever $\alpha \geq 1$, the class
  $\widetilde{\mathcal{A}}$ is non-empty and it contains the
  factorized solutions $(|\phi_t \rangle \langle \phi_t|^{\otimes
    k})_k$.
\end{theorem}

\begin{proof}
  The first part of the theorem is proved analogously as in the
  two-dimensional setting in Theorem
  \ref{uniqueness_Lambda2_2D}. Namely, we use the rescaling
  \eqref{rescaling} in three dimensions. We then finish the argument
  as before by using the spacetime estimate given in Proposition
  \ref{3D_spacetime_estimate}.

  For the second part of the theorem, we argue analogously as in the
  proof of \cite[Theorem 1.3]{GSS}. The analysis carries over to
  general tori once we recall the trilinear estimate given in
  \cite[Proposition 4.1]{Strunk}. More precisely, it is possible to
  set $\epsilon=1$ in \cite[Proposition 4.1]{Strunk} and deduce that
  there exists a universal constant $\delta_0>0$ such that for all
  dyadic integers $N_1,N_2,N_3$ with $N_1 \geq N_2 \geq N_3 \geq 1$,
  and for any finite interval $I$, it is the case that:
  \begin{equation}
    \label{Trilinear_Estimate}
    \begin{split}
   & \|P_{N_1} e^{it\Delta_Q} f_1 \cdot P_{N_2} e^{it\Delta_Q} f_2 \cdot P_{N_3} e^{it\Delta_Q}f_3 \|_{L^2(I \times \Lambda)} \\
    \lesssim{}& N_2 N_3\,\max\{\frac{N_3}{N_1},\frac{1}{N_2}\}^{\delta_0} \cdot \|P_{N_1}f_1\|_{L^2(\Lambda)} \cdot \|P_{N_2}f_2\|_{L^2(\Lambda)} \cdot \|P_{N_3}f_3\|_{L^2(\Lambda)}.
    \end{split}
  \end{equation}
  Here, $P_N$ denotes the projection to frequencies $|\xi| \sim
  N$. The implied constant depends only on the length of the interval
  $I$. We remark that \eqref{Trilinear_Estimate} can also be obtained from the more recent results in \cite{Bourgain_Demeter1,Killip_Visan}.
  In particular, from \eqref{Trilinear_Estimate}, it is possible
  to deduce the analogue of the trilinear estimate given in \cite[Proposition
  3.5]{HTT}, and following the arguments from \cite[Section 5]{GSS},
  it follows that the factorized solution $(|\phi_{Q,t} \rangle
  \langle \phi_{Q,t}|^{\otimes k})_k$ to the Gross-Pitaevskii
  hierarchy with modified Laplacian $\Delta_Q$ on $\Lambda$ in
  regularity $\alpha \geq 1$ belongs to the class $\mathcal{A}$. We
  then apply the scaling transformation \eqref{rescaling} and Lemma
  \ref{correspondence} in order to deduce that the factorized solution
  $(|\phi_t \rangle \langle \phi_t|^{\otimes k})_k$ to the
  Gross-Pitaevskii hierarchy on $\Lambda_2$ in regularity $\alpha \geq
  1$ belongs to the class $\widetilde{\mathcal{A}}$. Here, we note
  that $\phi_{Q,t}$ solves $i \partial_t \phi_{Q,t} + \Delta_Q
  \phi_{Q,t}=b_0 |\phi_{Q,t}|^2 \phi_{Q,t}$ on $\Lambda$ and $\phi_t$
  solves $i \partial_t \phi_t + \Delta \phi_t=b_0 |\phi_t|^2 \phi_t$
  on $\Lambda_2$, and the corresponding initial data are related by
  the scaling transformation \eqref{rescaling}.
\end{proof}

\begin{remark}
  The result of Theorem \ref{uniqueness_Lambda3_3D} also holds in the
  focusing setting.
\end{remark}

\begin{remark}
  \label{higher_dimensions_uniqueness}
  The conditional uniqueness results in Theorem
  \ref{uniqueness_Lambda2_2D} and Theorem \ref{uniqueness_Lambda3_3D}
  hold in general on $\Lambda_d$ for $d \geq 2$. More precisely, by
  using \eqref{d_dimensional_spacetime_estimate_bound} from Remark
  \ref{higher_dimensions}, it follows that solutions to \eqref{GP1}
  are unique in the class
  $\widetilde{\mathcal{A}}=\widetilde{\mathcal{A}}(\alpha)$, defined
  on $\Lambda_d$ for $\alpha>\frac{d-1}{2}$.
\end{remark}

\subsection{An unconditional uniqueness result and a rigorous
  derivation of the defocusing cubic NLS on $\Lambda_3$}
\label{An unconditional uniqueness result 3D}
It is also possible to prove an unconditional uniqueness result for
the Gross-Pitaevskii hierarchy on $\Lambda_3$ in a class of density
matrices in which the allowed range of regularity exponents is $\alpha
\geq 1$. As was noted in \cite{VS2}, this type of result does not
extend the conditional uniqueness result given in Theorem
\ref{uniqueness_Lambda3_3D}. However, as we will see below, it will
allow us to obtain a rigorous derivation of the defocusing cubic
nonlinear Schr\"{o}dinger equation on $\Lambda_3$.

Given $\alpha \in \mathbb{R}$, we denote by $\mathfrak{H}^{\alpha}$
the set of all $(\gamma^{(k)})_k \in \mathop{\bigoplus}_{k}
L^2(\Lambda_3^k \times \Lambda_3^k)$ such that:
\begin{itemize}
\item[i)] $\gamma^{(k)} \in L^2_{sym}(\Lambda_3^k \times \Lambda_3^k)$
  and
  $\gamma(\vec{x}_k,\vec{x}'_k)=\overline{\gamma^{(k)}(\vec{x}'_k;\vec{x}_k)}$
  for all $(\vec{x}_k,\vec{x}'_k)$ in $\Lambda_3^k \times
  \Lambda_3^k$.
\item[ii)] $S^{(k,\alpha)}\gamma^{(k)}$ belongs to the trace class on
  $L^2(\Lambda_3^k \times \Lambda_3^k)$.
\item[iii)] There exists $M>0$, which is independent of $k$, such that
  $Tr\big(|S^{(k,\alpha)}\gamma^{(k)}|\big) \leq M^{2k}$.
\end{itemize}

The unconditional uniqueness result that we prove is the following:
\begin{theorem}
  \label{unconditional_uniqueness_3D}
  Let $T>0$ be fixed. If $(\gamma^{(k)}(t))_k \in L^{\infty}_{t \in
    [0,T]} \mathfrak{H}^1$ is a mild solution to the Gross-Pitaevskii
  hierarchy on $\Lambda_3$, for which there exist $\Gamma_{N,t} \in
  L^2_{sym}(\Lambda_3^N \times \Lambda_3^N)$, which are non-negative
  as operators and whose trace is equal to $1$ such that:
  \begin{equation}
    \notag
    Tr_{k+1,\ldots,N} \Gamma_{N,t} \rightharpoonup^{*} \gamma^{(k)}(t)
  \end{equation}
  as $N$ tends to infinity in the weak-$*$ topology of the trace class
  on $L^2_{sym}(\Lambda_3^k)$. Then, the solution
  $(\gamma^{(k)}(t))_k$ is uniquely determined by the initial data
  $(\gamma_0^{(k)})_k$.
\end{theorem}
For the precise terminology and notation, we refer the reader to
\cite{ChHaPavSei} and \cite{VS2}.

In order to prove Theorem \ref{unconditional_uniqueness_3D}, one
argues analogously as in the proof of \cite[Theorem 4.6]{VS2}. In
other words, one applies the Weak Quantum de Finetti Theorem as in
\cite{ChHaPavSei}. The only difference from \cite{VS2} in the case of
the irrational torus $\Lambda_3$ is the fact that one has to use a
rescaled version of \eqref{Trilinear_Estimate} on $\Lambda_3$.
Such a trilinear estimate allows us to prove the analogue on
$\Lambda_3$ of \cite[Proposition 3.1]{VS2} and hence the analogues of
\cite[inequalities (41) and (42)] {VS2}, which were crucial in the
derivation analysis on $\mathbb{T}^3$. The further details of the proof of Theorem
\ref{unconditional_uniqueness_3D} are then the same as in the
setting of the classical torus \cite[Section 4]{VS2}. For a more
detailed discussion of this approach in the context of $\mathbb{R}^3$,
we refer the reader to \cite[Sections 4-8]{ChHaPavSei}.
By arguing as in \cite{VS2}, we can deduce the following derivation
result:
\begin{theorem}
  \label{NLS_Lambda3}
  Let $V:\mathbb{R}^3 \rightarrow \mathbb{R}$ be a non-negative,
  smooth, compactly supported function with $\int_{\mathbb{R}^3}
  V(x)\,dx=b_0>0$, and let $\beta \in (0,\frac{3}{5})$ be
  given. Suppose that $(\psi_N)_N \in \mathop{\bigoplus}_{N}
  L^2(\Lambda_3^N)$ satisfies the assumption of bounded energy per
  particle \eqref{Bounded energy per particle} and that of asymptotic
  factorization \eqref{Asymptotic factorization}.  Then, there exists
  a sequence $N_j \rightarrow \infty$ such that for all $t \in
  \mathbb{R}$ and $k \in \mathbb{N}$:
  \begin{equation}
    \notag
    Tr\big|\gamma^{(k)}_{N_j,t}-|\phi_t \rangle \langle \phi_t\big|^{\otimes k}\big| \rightarrow 0\,\,\,\mbox{as}\,\,j \rightarrow \infty,
  \end{equation}
  where $\phi_t$ solves the defocusing cubic nonlinear Schr\"{o}dinger
  equation on $\Lambda_3$ with initial data $\phi$:
  \begin{equation}
    \notag
    \begin{cases}
      i \partial_t \phi_t + \Delta \phi_t=b_0 |\phi_t|^2 \phi_t\\
      \phi_t \big|_{t=0}=\phi.
    \end{cases}
  \end{equation}
\end{theorem}
Let us note that the limiting arguments in \cite{EESY} carry over to
the setting of $\Lambda_3$. Namely, these arguments rely on Sobolev
embedding results and do not use any Diophantine properties. We can
then combine the analogue on $\Lambda_3$ of the result of \cite{EESY}
with the unconditional uniqueness result from Theorem
\ref{unconditional_uniqueness_3D}, and argue as in \cite[Section
5]{VS2} in order to deduce Theorem \ref{NLS_Lambda3}. We will omit the
details of the proof, and we will refer the interested reader to
\cite{VS2} for a more detailed discussion in the case of the classical
torus.

\begin{remark}
  It is possible to redo the analysis of this subsection in the
  two-dimensional setting if, instead of \cite[Proposition
  4.1]{Strunk}, we use \cite[Proposition 3.3]{Strunk}. In this way, we
  can also obtain a derivation of the defocusing cubic NLS on
  $\Lambda_3$, but by using an unconditional uniqueness result. We
  will not pursue this approach here. We recall that we have already
  obtained a derivation of the defocusing cubic NLS on $\Lambda_2$ in
  Theorem \ref{NLS_Lambda2} by using the conditional uniqueness result
  given in Theorem \ref{uniqueness_Lambda2_2D}.
\end{remark}

\section{A consequence of the spacetime estimates; local-in-time
  solutions to the Gross-Pitaevskii hierarchy on general tori}
\label{Local-in-time solutions}
In this section, let us recall that the spacetime estimates given in
Theorems \ref{2D_spacetime_estimate} and \ref{3D_spacetime_estimate},
as well as in \eqref{d_dimensional_spacetime_estimate_bound} of Remark
\ref{higher_dimensions} allow us to construct local-in-time solutions
to the Gross-Pitaevskii hierarchy with modified Laplacian $\Delta_Q$
on $\Lambda=\mathbb{T}^d$ for general $d \geq 2$. This is a direct
consequence of the truncation method from the work of T.~Chen and
Pavlovi\'{c} \cite{CP4}, which relies on the combinatorial boardgame
argument and on the spacetime estimate. Given $\alpha,\xi>0$, we
recall the definition of the space
$\mathcal{H}^{\alpha}_{\xi}(\mathbb{T}^d) \subseteq
\mathop{\bigoplus}_{k} L^2(\mathbb{T}^d \times \mathbb{T}^d)$, first
introduced in \cite{CP1}:
\begin{equation}
  \label{H_alpha_xi}
  \big\|(\gamma_0^{(k)})_k\big\|_{\mathcal{H}^{\alpha}_{\xi}(\mathbb{T}^d)}:=\sum_{k=1}^{\infty} \xi^k \cdot \|\gamma_0^{(k)}\|_{H^{\alpha}\big((\mathbb{T}^{d})^k \times (\mathbb{T}^{d})^k\big)}.
\end{equation}
In particular, we can deduce that the following result holds:

\begin{proposition}
  \label{local_existence_GP}
  Let $d \geq 2$ be given and we consider $\Lambda=\mathbb{T}^d$. Let
  us fix $\alpha>\frac{d-1}{2}$, and let
  $\alpha_0>\alpha$. Furthermore, let $\xi, \xi'>0$.  Suppose that
  $(\gamma_0^{(k)})_k \in
  \mathcal{H}^{\alpha_0}_{\xi'}(\Lambda)$. Then, if $\frac{\xi}{\xi'}$
  is sufficiently small depending on $d,\alpha,\alpha_0, \theta_1,
  \ldots, \theta_d$, there exists $T>0$ depending on
  $d,\alpha,\alpha_0,\xi,\xi', \theta_1, \ldots, \theta_d$ and
  $(\gamma^{(k)})_k=(\gamma^{(k)}(t))_k \in L^{\infty}_{[0,T]}
  \mathcal{H}^{\alpha}_{\xi}(\Lambda)$, such that for all $k \in
  \mathbb{N}$:
  \begin{equation}
    \label{local_solution_GP}
    \big\|\gamma^{(k)}(t)-\mathcal{U}_Q^{(k)}(t) \gamma_0^{(k)}+ib_0 \int_{0}^{t} \,\mathcal{U}_Q^{(k)}(t) B^{(k+1)} \gamma^{(k+1)}(s)\,ds\big\|_{L^{\infty}_{[0,T]} \mathcal{H}^{\alpha}_{\xi}(\Lambda)}=0.
  \end{equation}
\end{proposition}
We interpret \eqref{local_solution_GP} as $(\gamma^{(k)})_k$ being a
local-in-time solution of \eqref{GP}.

Furthermore, we can modify the definition given in \eqref{H_alpha_xi}
and define $\mathcal{H}^{\alpha}_{\xi}(\Lambda_d)$ on the general
$d$-dimensional torus $\Lambda_d$. By using the scaling
\eqref{rescaling} and Lemma \ref{correspondence}, we can deduce from
Proposition \ref{local_existence_GP} the following:

\begin{corollary}
  \label{local_existence_GP1}
  Let $\alpha,\alpha_0,\xi,\xi'$ be as in Proposition
  \ref{local_existence_GP} and let $(\widetilde{\gamma}_0^{(k)})_k \in
  \mathcal{H}^{\alpha_0}_{\xi}(\Lambda_d)$ be given. Then, for the $T$
  as in Proposition \ref{local_existence_GP}, there exists
  $(\widetilde{\gamma}^{(k)})_k=(\widetilde{\gamma}^{(k)}(t))_k \in
  L^{\infty}_{[0,T]} \mathcal{H}^{\alpha}_{\xi}(\Lambda_d)$, such that
  for all $k \in \mathbb{N}$:
  \begin{equation}
    \label{local_solution_GP1}
    \big\|\widetilde{\gamma}^{(k)}(t)-\mathcal{U}^{(k)}(t) \widetilde{\gamma}_0^{(k)}+ib_0 \int_{0}^{t} \,\mathcal{U}^{(k)}(t) B^{(k+1)} \widetilde{\gamma}^{(k+1)}(s)\,ds\big\|_{L^{\infty}_{[0,T]} \mathcal{H}^{\alpha}_{\xi}(\Lambda_d)}=0.
  \end{equation}
\end{corollary}
We interpret \eqref{local_solution_GP1} as
$(\widetilde{\gamma}^{(k)})_k$ being a local-in-time solution of
\eqref{GP1}.

\appendix
\section{Auxiliary results}
Let us recall the following result from Fourier analysis:
\begin{lemma}
  \label{Fourier_Analysis_Lemma}
  For $\delta>0$ sufficiently small, there exists $\zeta \in
  C_0^{\infty}(\mathbb{R})$, such that
  \begin{itemize}
  \item[i)] $\supp \,\zeta \subseteq [-\delta,\delta]$
  \item[ii)] $\widehat{\zeta} \geq 0$ on $\mathbb{R}$
  \item[iii)] $\widehat{\zeta} \geq 1$ on $[-1,1]$.
  \end{itemize}
\end{lemma}

\begin{proof}
  Let $\phi_1:=\frac{1}{2} \chi_{[-2\pi,2\pi]}$, where
  $\chi_{[-2\pi,2\pi]}$ denotes the characteristic function of the
  interval $[-2\pi,2\pi]$.  We compute:
  \begin{equation}
    \notag
    \widehat{\phi}_1(\xi)=\frac{\sin(2\pi \xi)}{\xi}.
  \end{equation}
  In particular, $\widehat{\phi}_1 \geq 2$ on $[-C,C]$, for some fixed
  $C>0$. By density, we can find $\phi_2 \in C_0^{\infty}(\mathbb{R})$
  such that $\|\phi_2-\phi_1\|_{L^1} \leq 1$.  Then
  $\|\widehat{\phi}_2-\widehat{\phi}_1\|_{L^{\infty}} \leq
  \|\phi_2-\phi_1\|_{L^1} \leq 1$, and hence $\widehat{\phi}_2 \geq 1$
  on $[-C,C]$.  Let:
  \begin{equation}
    \notag
    \phi_3:=\phi_2 * \mathcal{F}^{-1} \big(\overline{\widehat{\phi}_2}\big).
  \end{equation}
  Here, $\mathcal{F}^{-1}$ denotes the inverse Fourier transform.
  Then $\phi_3 \in C_0^{\infty}(\mathbb{R})$ and
  $\widehat{\phi}_3=|\widehat{\phi}_2|^2$ is non-negative on
  $\mathbb{R}$ and greater than or equal to $1$ on $[-C,C]$.  We now
  choose $m>0$ sufficiently small such that: \begin{equation} \notag
    \supp\,\phi_3\Big(\frac{\cdot}{m\delta}\Big) \subseteq
    [-\delta,\delta].
  \end{equation}
  Let $\zeta:=\frac{1}{m\delta} \phi_3(\frac{\cdot}{m\delta})$. Then
  $\zeta \in C_0^{\infty}(\mathbb{R}), \,\supp\,\zeta \subseteq
  [-\delta,\delta]$ and:
  \begin{equation}
    \notag
    \widehat{\zeta}(\xi)=\widehat{\phi}_3(m\delta\xi).
  \end{equation}
  Hence, $\zeta \geq 0$ on $\mathbb{R}$ and $\widehat{\zeta} \geq 1$
  whenever $|\xi| \leq \frac{C}{m\delta}$. We then choose $\delta>0$
  sufficiently small so that $\frac{C}{m\delta} \geq 1$.
\end{proof}

\end{document}